\documentclass[a4paper,authoryear,12pt]{article}
\usepackage{}
\usepackage{amsmath}
\usepackage{cite}
\usepackage[latin1]{inputenc}
\usepackage{amsmath,amsthm,amssymb}
\usepackage{enumerate}
\usepackage[dvipdfm,colorlinks,linkcolor=blue,hyperindex,pdfstartview=FitH,plainpages=false,CJKbookmarks]{hyperref}

\newcommand{\Rmnum}[1]{\expandafter\@slowromancap\romannumeral #1@}

\newcommand{\Om}{\Omega}

\newcommand{\ds}{\displaystyle}

\makeatother
\title{Emergence of lager densities in chemotaxis system with indirect signal production and non-radial symmetry case\footnote{This work is supported by the Scientific Research Fund (YS304221937) and the Young Doctor Program of Zhejiang Normal University (ZZ323205020520013068)}}
\author {Guangyu Xu\thanks{Corresponding author: guangyuswu@126.com,\ xuguangyu@zjnu.edu.cn}\\
\small College of Mathematics and Computer Science, Zhejiang Normal University,\\
\small Jinhua 321004, P.R. China.}
\date{}
\newtheorem{theorem}{Theorem}[section]

\newtheorem{definition}{Definition}[section]
\newtheorem{lemma}{Lemma}[section]
\newtheorem{remark}{Remark}
\newtheorem{corollary}{Corollary}[section]
\begin{document}
\baselineskip20pt \maketitle
\renewcommand{\theequation}{\arabic{section}.\arabic{equation}}
\catcode`@=11 \@addtoreset{equation}{section} \catcode`@=12
\begin{abstract}
  This paper deals with the classical solution of the following chemotaxis system with generalized logistic growth and indirect signal production
  \begin{eqnarray}\label{00}
  \left\{
  \begin{array}{llll}
  u_t=\epsilon\Delta u-\nabla\cdot(u\nabla v)+ru-\mu u^\theta, &\\
  0=d_1\Delta v-\beta v+\alpha w, &\\
  0=d_2\Delta w-\delta w+\gamma u, &
  \end{array}
  \right.
  \end{eqnarray}
  and the so-called strong $W^{1, q}(\Om)$-solution of hyperbolic-elliptic-elliptic model
  \begin{eqnarray}\label{01}
  \left\{
  \begin{array}{llll}
  u_t=-\nabla\cdot(u\nabla v)+ru-\mu u^\theta, &\\
  0=d_1\Delta v-\beta v+\alpha w, &\\
  0=d_2\Delta w-\delta w+\gamma u, &
  \end{array}
  \right.
  \end{eqnarray}
  in arbitrary bounded domain $\Om\subset\mathbb{R}^n$, $n\geq1$, where $r, \mu, d_1, d_2, \alpha, \beta, \gamma, \delta>0$ and $\theta>1$. Via applying the viscosity vanishing method, we first prove that the classical solution of \eqref{00} will converge to the strong $W^{1, q}(\Om)$-solution of \eqref{01} as $\epsilon\rightarrow0$. After structuring the local well-pose of \eqref{01}, we find that the strong $W^{1, q}(\Om)$-solution will blow up in finite time with non-radial symmetry setting if $\Om$ is a bounded convex domain, $\theta\in(1, 2]$, and the initial data is suitable large. Moreover, for any positive constant $M$ and the classical solution of \eqref{00}, if we add another hypothesis that there exists positive constant $\epsilon_0(M)$ with $\epsilon\in(0,\ \epsilon_0(M))$, then the classical solution of \eqref{00} can exceed arbitrarily large finite value in the sense: one can find some points $\left(\tilde{x}, \tilde{t}\right)$ such that $u(\tilde{x}, \tilde{t})>M$.
  \end{abstract}
{\bf AMS(2010) Subject Classification}: 35B44; 35K55; 35K40; 92C17.

\section{Introduction}
\hspace*{\parindent}

In this paper, we consider following initial boundary problem to chemotaxis system with indirect signal production
\begin{eqnarray}{\label{1}}
\left\{
\begin{array}{llll}
u_t=\epsilon\Delta u-\nabla\cdot(u\nabla v)+ru-\mu u^\theta, &x\in \Omega,\ t>0,\\
0=d_1\Delta v-\beta v+\alpha w, &x\in \Omega,\ t>0,\\
0=d_2\Delta w-\delta w+\gamma u, &x\in \Omega,\ t>0,\\
\ds\frac{\partial u}{\partial\nu }=\frac{\partial v}{\partial \nu}=\frac{\partial w}{\partial\nu }=0, &x\in\partial\Omega,\ t>0,\\
u(x, 0)=u_0, &x\in\Omega,
\end{array}
\right.
\end{eqnarray}
where $\Omega\subset \mathbb{R}^n$ is an bounded domain with smooth boundary $\partial\Om, n\geq1,\ \frac{\partial}{\partial\nu}$ denotes the derivative with respect to the outer normal on $\partial\Omega$, a no-flux condition is implied on boundary so that the ecosystem is closed to the exterior environment. The parameters satisfy
\begin{equation}\label{csdy}
\epsilon, r, \mu, d_1, d_2, \alpha, \beta, \gamma, \delta>0, \theta>1.
\end{equation}
Initial data $u_0$ is positive and satisfy
\begin{equation}\label{cszjs}
u_0\in W^{1, q}(\Om), q>n.
\end{equation}
Chemotaxis model with indirect signal production mechanism has been proposed in \cite{Strohm2013} to describe the spread and aggregative behaviors of the mountain pine beetle (MPB) in a forest habitat considered negligibly thin in its vertical dimension. The model describes the space and time evolutions of the densities of flying MPB, nesting MPB and the concentration of the beetle pheromone, they are represented by $u=u(x, t), w=w(x,t)$ and $v=v(x, t)$, respectively. Problem \eqref{1} can be viewed as a generalization of the classical Keller-Segel system, and the main feature of the corresponding {\color{blue} system} is that the signal productions occur in an indirect process, with first $u$ producing the third quantity $w$ and the latter being exclusively responsible for the release of $v$.

\subsection{Background}

Chemotaxis is a biological phenomenon describing the movement of individuals in the direction of increasing chemical concentration spread in the environment where they reside. The classical Keller-Segel model was proposed in \cite{keller1970initiation} to describe the aggregation of cellular slime molds Dictyostelium discoideum and in \cite{keller1971model} to describe the wave propagation of bacterial chemotaxis, where the chemical signals inducing a bias in the motion of the species is produced directly by the species itself. However, in a large number of realistic situations the mechanisms of signal production may be substantially more complicated, see for example \cite{Dillon1994,dingmeng2019,Painter2000,tello2016predator,Tuval2005}.

For following chemotaxis model with indirect signal production mechanism,
\begin{eqnarray}{\label{1jl}}
\left\{
\begin{array}{llll}
u_t=\Delta u-\nabla\cdot(u\nabla v)+f(u), \\
\tau v_t=\Delta v-v+w, \\
\tau w_t=d\Delta w-\delta w+u,
\end{array}
\right.
\end{eqnarray}
when $f(u)=0$, Tao and Winkler \cite{tao2017Critical} considered the model in the unit disk and radially symmetric setting with $d=0$ and the second equation replaced by $0=\Delta v-\frac{1}{|\Om|}\int_\Om wdx+w$, for the solution to this parabolic-elliptic-ODE system they found the mass threshold property refer to blow-up in infinite time rather than in finite time, which is surprisingly different from the case of classical Keller-Segel model. Later on, Lauren\c{c}ot \cite{Laurencot2019} considered model \eqref{1jl} with $d=0$ in arbitrary two-dimensional bounded domains. The author actually showed that the system possess a Liapunov functional and that the properties of this Liapunov functional provide insight on the boundedness or unboundedness of the solutions. For $d>0$, Fujie and Senba \cite{fujie2017application} studied system \eqref{1jl} with a fully parabolic type indirect signal production mechanism. Unlike the observations obtained in \cite{herrero1997blow,Horstmann2002,Horstmann2001,nagai2001blowup,Nagai1997Application} for the classical Keller-Segel system and in \cite{Laurencot2019,tao2017Critical} for the corresponding ODE-type indirect signal production system, therein the critical mass phenomenon occurs in the two-dimensional situation, the authors in \cite{fujie2017application} found that for the corresponding PDE-type indirect signal production model \eqref{1jl} the four-dimensional setting is the critical case and the critical mass is $m_c:=\frac{(8\pi)^2}{\chi}$. Specifically speaking, when $ n\leq 3$ one can invoke smoothing effects of diffusion terms of each equation of \eqref{1jl} and derive sufficiently high regularity estimates for solutions guaranteeing global existence from the mass conservation law. However in the four-dimensional setting the smoothing effect does not work enough and solutions remain bounded whenever the domain and initial data are radially symmetric such that $\int_\Om u_0<m_c$. Bai and Liu \cite{bai2018new} obtained an extension as well as a simplified proof for this boundedness result. Whereas the system under mixed boundary conditions in a smooth bounded convex domain possesses some solutions which blows up in finite or infinite time provided that $\int_\Om u_0>m_c$ \cite{fujie2019blowup}.

From the mentioned papers we know such indirect signal production mechanisms may delay or prevent the spontaneous singularity formation. It is widely known that the repulsion effect in chemotaxis model actually also benefits the global boundedness of solutions. The following attraction-repulsion chemotaxis system was proposed in \cite{luca2003chemotactic} to describe the aggregation of microglia observed in Alzheimer's disease and in \cite{painter2002volume} to describe the quorum effect in the chemotactic process,
\begin{eqnarray}{\label{xp1}}
\left\{
\begin{array}{llll}
u_t=\Delta u-\chi \nabla\cdot(u\nabla v)+\xi \nabla\cdot(u\nabla w)+f(u),\quad &x\in \Omega,\quad t>0,\\
\tau v_t=\Delta v+\tilde{\alpha} u-\tilde{\beta} v,\quad &x\in \Omega,\quad t>0,\\
\tau w_t=\Delta w+\tilde{\gamma} u-\tilde{\delta} w,\quad &x\in\Omega,\quad t>0.
\end{array}
\right.
\end{eqnarray}
For the global boundedness results with $n\leq3$, one can see \cite{Jin2015Asymptotic,jin2015boundedness,jin2019global,Liu2012Classical,lin2015large,liu2015global,tao2013competing}, see also \cite{Jin2016Boundedness} for the critical mass phenomenon with $n=2$ in an attraction-repulsion chemotaxis system with the attraction dominating case. Tao and Wang \cite{tao2013competing} considered problem \eqref{xp1} with multi-dimensional case and $\tilde{\beta}=\tilde{\delta}$, they studied the global solvability, boundedness, blow-up, existence of non-trivial stationary solutions and asymptotic behavior of the system for various ranges of parameter values. For problem \eqref{xp1} with $\tau=0, \chi\tilde{\alpha}=\xi\tilde{\gamma}, \tilde{\beta}\not=\tilde{\delta}$ and $f$ satisfying
\begin{equation*}
f(s)\leq a-bs^r \ \ \text{for all}\ \ s\geq0 \ \ \text{with some}\ \ a\geq0,\ b>0 \ \text{and}\ r\geq 1,
\end{equation*}
Li and Xiang \cite{li2016attraction} obtained that the boundedness of solutions was determined specially under the assumption that $n\geq2, r>\frac{1}{2}\left(\sqrt{n^2+4n}-n+2\right)$. Thus, owing to the effect of repulsion, the exponent $r$ is allowed to take values less than 2 such that the solution remains uniformly bounded in time. Wang et al. \cite{wang2018positive} proved the solution to the corresponding fully parabolic system is globally bounded when
\begin{equation*}
n\leq3\ \ \text{or}\ \ r>r_n:=\left\{\frac{n+2}{4},\ \frac{n\sqrt{n^2+6n+17}-n^2-3n+4}{4} \right\}\ \ \text{with}\ \ n\geq2.
\end{equation*}
It is clearly $r_4=\frac{3}{2}$. More recently, Li and Wang \cite{li2018boundedness} found that problem \eqref{xp1} with $\tau=0, \chi\tilde{\alpha}=\xi\tilde{\gamma}$ and $\tilde{\beta}\not=\tilde{\delta}$ possesses a globally bounded and classical solution if $r=\frac{3}{2}$.

There are also many works about problem \eqref{1jl} with $f(u)$ being a damping source term in order to improve the system's consistency with biological reality or adapt it to complex biological situations. It is well known that an appropriate logistic damping can prevent blow-up of solutions to the classical Keller-Segel system
\begin{eqnarray}{\label{1.6}}
\left\{
\begin{array}{llll}
U_t=\epsilon\Delta U-\chi\nabla\cdot(U\nabla V)+\tilde{r}U-\tilde{\mu} U^2,\quad &x\in \Omega,\quad t>0,\\
\kappa V_t=\Delta V-V+U,\quad &x\in\Omega,\quad t>0.
\end{array}
\right.
\end{eqnarray}
When $\kappa=0$, if either $n\leq 2$ and $\tilde{\mu}>0$ is arbitrary, or if $n\geq3$ and $\tilde{\mu}$ suitable lager, then for all reasonably regular initial data model \eqref{1.6} possesses a globally bounded classical solution \cite{tello2007chemotaxis}. Moreover, the effect of logistic damping is stronger than that of chemotactic aggregation when $\tilde{\mu}\geq\frac{n-2}{n}\chi$ \cite{kang2016blowup,xiang2019dynamics,tello2007chemotaxis}. For the case of $\kappa=1$, Winkler \cite{winkler2010boundedness} showed that if $\mu$ is sufficiently large then problem \eqref{1.6} possesses a unique bounded solution, and there are also further progresses in this direction, one can see \cite{lin2016global,xiang2018strong}. Especially, for any $\tilde{\mu}>0$ when $n=3$ and the domain $\Om$ is convex, it is known from \cite{Chemotaxis2} that blow-up is not possible in the weak solution sense. In drastic contrast to \eqref{1.6}, Hu and Tao \cite{hu2016to} proved an arbitrarily small quadratic degradation term is sufficient to suppress any blow-up phenomenon in problem \eqref{1jl} when $n=3$ and $d=0$, they further showed the solution exponentially stabilizes to the constant stationary solution if $\mu$ suitable lager. The boundedness result got in \cite{hu2016to} had been improved in \cite{lia2018boundedness} to a general case, wherein $f(u)=u-u^\theta$ is considered and the solution exists bounded when $n\geq2$ and $\theta>\frac{n}{2}$. Ren and Liu \cite{Ren2020} proved that the solution is bounded if $n\geq3$ and $\theta=\frac{n}{2}$. Lv and Wang \cite{lv2020global} studied a chemotaxis system with signal-dependent motility, indirect signal production and generalized logistic source, and they established the global existence of the solution under some assumptions on the parameters of the logistic source. One can refer to  \cite{fuest2019analysis,liu2020global,noda2019global,Nakaguchi2020,qiu2020boundedness,qiu2018boundedness,surulescu2019does,tang2020,zheng2020boundedness} for further extensions of related chemotaxis models involving ODE-type indirect signal production mechanism and damping source.

For the PDE-type indirect signal production systems with damping source, Zhang el at. \cite{zhang2019large} claimed that if $f(u)=\mu(u-u^\theta)$ and $\theta>\frac{n}{4}+\frac{1}{2}$, then the solution is globally bounded. This boundedness conclusion still holds for a chemotaxis system with indirect signal production and rotational sensitivity \cite{dong2021global}. The authors in \cite{wangwei2019,zhang2020asymptotic} considered the corresponding quasilinear fully parabolic chemotaxis system with indirect signal production and logistic source, and they obtained the global boundedness of solutions under some assumptions on parameters of self-diffusion, cross-diffusion and logistic source.

\subsection{Main results}

Let us return the attention to problem \eqref{1} and give the main results of this paper. From the existing papers as mentioned above, we know the indirect production mechanism is helpful for the global existence of solutions, and it is well known that the damping source is also a main method to avoid the chemotactic collapse. Problem \eqref{1} is provided with these two characteristics at the same time, as far as we know, thus there is no any blow-up result or unbound phenomenon for it. In fact, even for the classical chemotaxis-growth system \eqref{1.6}, the analysis of this aspect is incomplete.

When $\kappa=0, \tilde{\mu}\in(0, 1)$, Winkler \cite{Winkler2014How} considered dynamical behaviors of model \eqref{1.6} with $n=\chi=1$, and he claimed that if $\|u_0\|_{L^p(\Om)}$ is sufficiently large, where $p>\frac{1}{1-\tilde{\mu}}$, then there exists $T>0$ such that to each $M>0$ there corresponds some $\epsilon_0(M)>0$ with the property that for any $\epsilon\in(0, \epsilon_0(M))$ one can find $t_\epsilon\in(0, T)$ and $x_\epsilon\in\Om$ such that the classically solution satisfies
\begin{equation}\label{xsz}
U(x_\epsilon, t_\epsilon)>M.
\end{equation}
The authors in \cite{Chemotaxis} extended this result to higher dimensional and radially symmetric case. In addition, under the radially symmetric assumptions on physical region and initial data we \cite{xu2020carrying} investigated the chemotaxis models with Lotka-Volterra type competitive species and gained the property \eqref{xsz} for the corresponding solutions recently. For more results on the property like \eqref{xsz} with related fully parabolic chemotaxis models under the radially symmetric assumptions, one can see \cite{li2019emergence,2020The,winkler2017emergence,wang2019the,2020Theyang}.

However, to the best of our knowledge, there is no any result about such transient growth phenomena for problem \eqref{1}. In this paper, we thus aim to study this dynamic behavior of solution to problem \eqref{1} with non-radial symmetry assumption. Unlike the outcomes for the classical Keller-Segel system \eqref{1.6} with $\tau=0$ obtained in \cite{Chemotaxis,Winkler2014How}, we shall establish the corresponding conclusions for problem \eqref{1} in a general bound domain, so it is an improvement on the basis of \cite{Chemotaxis,Winkler2014How}.

Say concretely, we shall consider problem \eqref{1} with applying the viscosity vanishing method. In this end, we need further study the corresponding hyperbolic-elliptic-elliptic model which may be seen as an auxiliary problem and can be formally obtained by letting $\epsilon\rightarrow 0$,
\begin{equation}\label{2}
\left\{
\begin{array}{ll}
u_t=-\nabla\cdot(u\nabla v)+ru-\mu u^\theta,\quad &x\in \Omega,\quad t>0,\\
0=d_1\Delta v-\beta v+\alpha w,\quad &x\in \Omega,\quad t>0,\\
0=d_2\Delta w-\delta w+\gamma u,\quad &x\in\Omega,\quad t>0,\\
\ds\frac{\partial u}{\partial\nu }=\frac{\partial v}{\partial \nu}=\frac{\partial w}{\partial\nu }=0,\quad &x\in\partial\Omega,\quad t>0,\\
u(x, 0)=u_0\, \,\quad &x\in\Omega.
\end{array}
\right.
\end{equation}
Then, we consider the so-called strong $W^{1,q}$-solution $(u, v, w)$ (see \cite{Winkler2014How}) to problem \eqref{2}, which can be defined as following sense:

\begin{definition}\label{21}
Let $r, \mu, d_1, \alpha, \beta, d_2, \gamma, \delta>0$ and $\theta>1, T^*\in(0, +\infty], q>n$. A strong $W^{1,q}$-solution of problem \eqref{2} in $\Omega\times(0, T^*)$ is the nonnegative function $u\in C(\bar\Omega\times[0, T^*))\cap L_{loc}^\infty([0, T^*); W^{1,q}(\Omega))$ and $v, w\in C^2(\bar\Omega\times[0, T^*))$
such that $u$ satisfies
\begin{equation}\begin{split}\label{13}
-\int_0^{T^*}\int_\Om u\varphi_t-\int_\Om u_0\varphi(\cdot, 0)=\int_0^{T^*}\int_\Om u\nabla v\nabla\varphi+\int_0^{T^*}\int_\Omega (ru-\mu u^\theta)\varphi
\end{split}\end{equation}
for any $\varphi\in C_0^\infty([0, T^*); \bar\Omega)$, and $v, w$ classically solves
\begin{eqnarray*}
\left\{
\begin{array}{ll}
0=d_1\Delta v-\beta v+\alpha w,\quad &(x, t)\in\Omega\times(0, T^*)\\
0=d_2\Delta w-\delta w+\gamma u,\quad &(x, t)\in\Omega\times(0, T^*)\\
\ds\frac{\partial v}{\partial \nu}=\frac{\partial w}{\partial \nu}=0,\quad &(x, t)\in\partial\Omega\times(0, T^*).
\end{array}
\right.
\end{eqnarray*}
\end{definition}

In fact, with means of a compactness argument we shall prove that the solution of problem \eqref{1} converges to the strong $W^{1, q}$-solution of \eqref{2} as $\epsilon\rightarrow 0$. More importantly, by means of a logarithmic Sobolev inequality we can built the result in a general physical domain. Specifically, we can ensure that problem \eqref{2} is locally well posed in the following sense.

\begin{theorem}\label{jbczs}
Let $\Om$ be a bounded region and $u_0\in W^{1, q}(\Om), q>n$. Then there exists a constant $T^*\in(0, +\infty]$ and problem \eqref{2} admits an unique strong $W^{1,q}$-solution $(u, v, w)$ with $u\in C(\bar\Omega\times[0, T^*))\cap L_{loc}^\infty([0, T^*); W^{1,q}(\Omega))$ and $v, w\in C^{2, 0}(\bar\Omega\times[0, T^*))$. Moreover,
\begin{equation*}
\mbox{either}\ \ T^*=+\infty \ \ \mbox{or}\ \ \limsup_{t\rightarrow T^*} \|u(\cdot, t)\|_{L^\infty(\Om)}=+\infty,
\end{equation*}
and for any $t\in(0, T^*)$ we have
\begin{equation}\begin{split}\label{hy}
\beta\int_\Om v(\cdot, t)=\alpha\int_\Om w(\cdot, t)&=\frac{\alpha\gamma}{\delta}\int_\Omega u(\cdot, t)\leq\frac{\alpha\gamma}{\delta}m_1,
\end{split}\end{equation}
where
\begin{equation}\label{m12d}
m_1:=\max\left\{\frac{1}{|\Om|}\int_\Om u_0,\ \ \left(\frac{r}{\mu}\right)^{\frac{1}{\theta-1}}\right\}.
\end{equation}
\end{theorem}

Although the strong $W^{1,q}$-solution exists boundedness with $L^1(\Om)$-norm for all $(0, T^*)$ and problem \eqref{2} possesses the indirect signal production mechanism along with the logistic source, we will prove that there exists still a finite time blow-up phenomenon with $n\geq1$. Due to the lack of free diffusion of $u$, it should be pointed out that the blow-up result for problem \eqref{2} holds with non-radial symmetry setting and for all $n\geq1$. Concretely, we have following result.

\begin{theorem}\label{33}
Let $n\geq1$, $\Om$ be a bounded convex region, the parameters satisfy \eqref{csdy} with $\theta\in(1, 2]$. Then one can find suitably large constants $c_1, c_2>0$ with the following property: If the positive initial data $u_0\in W^{1, q}(\Om), q>n$ is suitably large such that
\begin{equation}\label{theta}
\|u_0\|^\theta_{L^\theta(\Om)}>c_1\cdot\max\left\{1,\ m_1^{\theta}\right\}+c_2,
\end{equation}
where $m_1$ is the positive constant given by \eqref{m12d}, then the strong $W^{1,q}$-solution of problem \eqref{2} blows up in finite time.
\end{theorem}

\begin{remark}
To the best of our knowledge, for the chemotaxis systems with both of the ODE- and PDE-types indirect signal production as mentioned above, the existence of finite time blow-up solutions has been still open. For all $n\geq1$, Theorem \ref{33} provides a finite time blow-up result for problem \eqref{2} with non-radial symmetry setting although there exist the logistic source, to a certain degree, this phenomenon reflects how important the free diffusion of $u$ is for the dynamic behavior of the solution to the corresponding chemotaxis model.
\end{remark}

We then obtain the transient growth phenomena to problem \eqref{1} by means of a continuous dependence argument.

\begin{theorem}\label{d1q}
Let $n\geq1$, $\Om$ be a bounded convex region, $M$ be an arbitrary positive constant, the parameters satisfy \eqref{csdy} with $\theta\in(1, 2]$. If there exists a positive constant $\epsilon_0(M)$ such that $\epsilon\in(0,\ \varepsilon_0(M))$ and the initial data satisfy \eqref{cszjs} and \eqref{theta}, then for the classical solution of problem \eqref{1} one can find some points $\left(\tilde{x}, \tilde{t}\right)$ such that
\begin{equation}\label{cl}
u\left(\tilde{x}, \tilde{t}\right)>M.
\end{equation}
\end{theorem}

\begin{remark}
We mainly focus on the effects of $\epsilon$ and the size of $u_0$ on the behaviors of the classical solution of problem \eqref{1}, and the main contribution of Theorem \ref{d1q} is that it provide the property \eqref{cl} to the chemotaxis system with indirect signal production mechanism: In a general convex domain, if the initial data is sufficiently lager with the fixed coefficient, if further the species do not diffuse too fast, i.e. $\epsilon$ is suitable small, then we can make sure the existence of some time up to which any threshold of the species density will be surpassed.
\end{remark}

\begin{remark}\label{r3}
Let us explain the assumption on $\epsilon$ in Theorem \ref{d1q}. In order to give a convergence between the classical solution of problem \eqref{1} and the $W^{1, q}$-solution of problem \eqref{2} (see Lemma \ref{25}), we need a boundedness (independent of $\epsilon$) of $u$ in $L^\infty([0, T]; W^{1, q}(\Om))$ (see Lemma \ref{144444}), in this process we have to give a upper bound on $\epsilon$ to handle the boundary integration. Moreover, as the same requirements in \cite[Lemma 3.4]{Winkler2014How} and \cite[Lemma 3.10]{Chemotaxis}, an upper bound of $\epsilon$ is also required for problem \eqref{1} to get an uniformly bounded (independent of $\epsilon$) estimate for the time derivative. Furthermore, from a biological point of view, this smallness assumption on the free diffusion rate $\epsilon$ obviously will be of benefit to the aggregation behavior of the system. On the other hand, although $\theta\in(1, 2]$, due to the smallness assumption on $\epsilon$, we can see that the property \eqref{cl} does not conflicts any boundedness results about problem \eqref{1} mentioned above, since the free diffusion rate considered therein is fixed for a positive constant.
\end{remark}

\begin{remark}
Theorem \ref{d1q} extends the classical Keller-Segel growth-type model considered in \cite{Chemotaxis,Winkler2014How} to problem \eqref{1}, and removes the radially symmetric assumptions on $\Om$ and initial data which were needed in \cite{Chemotaxis,Winkler2014How}.
\end{remark}

Assume $v(x,t):=\chi z(x, t)-\xi w(x, t), \xi\gamma=\chi\alpha$, then the following attraction-repulsion chemotaxis system
\begin{eqnarray}{\label{xp1ca}}
\left\{
\begin{array}{llll}
u_t=\epsilon\Delta u-\chi \nabla\cdot(u\nabla z)+\xi \nabla\cdot(u\nabla w)+ru-\mu u^\theta,\  &x\in \Omega,\ t>0,\\
0=d_1\Delta z+\alpha z-\beta z,\  &x\in \Omega,\  t>0,\\
0=d_2\Delta w+\gamma u-\delta w,\  &x\in\Omega,\  t>0,
\end{array}
\right.
\end{eqnarray}
becomes to
\begin{eqnarray*}
\left\{
\begin{array}{llll}
u_t=\epsilon\Delta u-\nabla\cdot(u\nabla v)+ru-\mu u^\theta, &x\in \Omega,\ t>0,\\
0=d_1\Delta v-\beta v+\bar{\alpha} w, &x\in \Omega,\ t>0,\\
0=d_2\Delta w-\delta w+\gamma u, &x\in \Omega,\ t>0,
\end{array}
\right.
\end{eqnarray*}
where $\bar{\alpha}=\xi(\frac{d_1\delta}{d_2}-\beta)$, so as a by-product, we can see from Theorem \ref{d1q} that the transient growth phenomena is also true for the attraction-repulsion chemotaxis model.

\begin{corollary}
For the classical solution of the corresponding initial-boundary value problem \eqref{xp1ca}, let $n\geq1$, $\Om$ be a bounded convex region, $\theta\in(1, 2], \chi, \xi, r, \mu, d_1, d_2, \beta, \delta, \gamma>0$ with
\begin{equation*}
\xi\gamma=\chi\alpha,\  \ d_1\delta\not=d_2\beta,
\end{equation*}
if for any positive constant $M$, there exists a positive constant $\epsilon_0(M)$ such that $\epsilon\in(0,\ \varepsilon_0(M))$ and the initial data satisfy \eqref{cszjs} and \eqref{theta}, in this setting the positive constants $c_1, c_2$ depends on $\chi, \xi, d_1, d_2, \beta, \gamma, \delta, \mu, n$ and $\Om$, then the property \eqref{cl} is valid.
\end{corollary}

\begin{remark}\label{rv}
Also, like in Remark \ref{r3}, in view of the smallness assumption on $\epsilon$, we point out that property \eqref{cl} for the attraction-repulsion chemotaxis model does not contradicts any boundedness results about the corresponding problem \eqref{xp1ca}.
\end{remark}

We would like to mention the results got in \cite{winkler2011blow,winkler2018finite} which convey to us that chemotactic collapses are still possible even in the presence of superlinear growth restrictions. How this singularity property can be constructed for problem \eqref{1}? We leave this open problem as a future investigation.

The rest of this paper is organized as follows. In Section 2, we give some known results concerning local existence of solution to problem \eqref{1} as well as some useful preliminary Lemmata. In Section 3, we consider the local well-pose for the corresponding hyperbolic-elliptic-elliptic model \eqref{2}, and then give the proofs of our main theorems.

\section{Preliminaries}

As a preliminary, we state firstly the following local well-posedness of classical solution to problem \eqref{1}. As to problem \eqref{1}, the local existence and uniqueness of solutions are established by Banach's fixed point theorem and well-known parabolic regularity theory in the same way as in for example \cite{Bellomo2015Toward,horstmann20031970,winkler2010boundedness,wrzosek2006long-time}.

\begin{lemma}{\label{3.1}}
For any parameters satisfy \eqref{csdy} and positive initial value satisfying \eqref{cszjs}, there exists a $T_{\max}\in (0, +\infty]$ and unique positive classical solution to problem \eqref{1} in $\Omega\times(0, T_{\max})$. We further have the extensibility criterion: either $T_{\max}=+\infty$, or
\begin{equation*}
\lim_{t\rightarrow T_{\max}}\|u(\cdot,t)\|_{L^{\infty}(\Omega)}=+\infty.
\end{equation*}
\end{lemma}

Throughout the sequel, it will not be mentioned any further that $C>0$ is a generical constant although they can change line by line.

We recall following maximal regularity of elliptic problem, which can be found in \cite[Lemma 4]{Chemotaxis}, or see \cite{4,gilbarg2015elliptic} for its' detailed proof.

\begin{lemma}\label{2ly}
For any $f\in C(\bar\Om)$, the solution $\phi$ of following problem
\begin{eqnarray*}
\left\{
\begin{array}{ll}
-\Delta\phi+\phi=f,\quad &x\in \Omega,\\
\ds\frac{\partial \phi}{\partial\nu }=0,\quad &x\in\partial\Omega,
\end{array}
\right.
\end{eqnarray*}
satisfies $\|\phi\|_{L^{p}(\Om)}\leq\|f\|_{L^{p}(\Om)}$ for all $p\in[1, \infty]$. Moreover, for any $q\geq 1, \iota>0$, there is a constant $C_*>0$ such that $\|v\|_{W^{2,q}(\Om)}\leq C_*\|f\|_{L^{q}(\Om)}$ and $\|v\|_{C^{2+\iota}(\Om)}\leq C_*\|f\|_{C^{\iota}(\Om)}$.
\end{lemma}

\begin{lemma}\label{xs3}
Let $p>1$ and $\eta_1, \eta_2>0$ be three constants, then for any classical solution $(v, w)$ of the problem
\begin{eqnarray*}
\left\{
\begin{array}{ll}
0=d_1\Delta v-\beta v+\alpha w,\quad &x\in\Omega,\\
0=d_2\Delta w-\delta w+\gamma u,\quad &x\in\Omega,\\
\ds\frac{\partial v}{\partial \nu}=\frac{\partial w}{\partial \nu}=0,\quad &x\in\partial\Omega,
\end{array}
\right.
\end{eqnarray*}
there are two positive constants $C_1=C(\eta_1, p, d_1, \alpha, \beta), C_2=C(\eta_2, p, d_2, \delta, \gamma)>0$ such that
\begin{equation*}\begin{split}
\int_\Om v^{p+1}\leq\eta_1\int_\Omega w^{p+1}+C_1\left(\int_\Om w\right)^{p+1},\\
\int_\Om w^{p+1}\leq\eta_2\int_\Omega u^{p+1}+C_2\left(\int_\Om u\right)^{p+1}.
\end{split}\end{equation*}
\end{lemma}

\begin{proof}
The conclusion can be proved by the same arguments in \cite[Lemma 2.2]{Winkler2014How} or \cite[Lemma 2.2]{Chemotaxis}, so we omit the details here.
\end{proof}

In order to manage the boundary integration, the following three lemmata are necessary for us, they can be found in \cite{ishida2014boundedness}.

\begin{lemma}\label{bjt}
Let $\Om$ be a bounded domain in $\mathbb{R}^n$ with smooth boundary. If $\psi\in C^2(\bar\Om)$ satisfies $\frac{\partial\psi}{\partial\nu}=0$, then $\frac{\partial|\nabla\psi|^2}{\partial\nu}\leq C_\Om|\nabla\psi|^2$, where $C_\Om>0$ is a constant depending only on the curvatures of $\partial\Om$.
\end{lemma}

\begin{lemma}\label{bjt2}
Let $\Om$ be a bounded domain in $\mathbb{R}^n$ with smooth boundary and $\tilde{r}>0$ be a constant. Then $W^{\tilde{r}, 2}(\partial\Om)\hookrightarrow L^2(\partial\Om)$ is a compact embedding and there exists a linear and bounded map from $W^{\tilde{r}+\frac{1}{2}, 2}(\Om)$ onto $W^{\tilde{r}, 2}(\partial\Om)$.
\end{lemma}

\begin{lemma}\label{bjt3}
Let $n\in\mathbb{N}$ and $s, s_1\geq1$. Assume that $p>0$ and $a\in(0,1)$ satisfy
\begin{equation*}\begin{split}
\frac{1}{2}-\frac{p}{n}=(1-a)\frac{s_1}{s}+a\left(\frac{1}{2}-\frac{1}{n}\right)\ \mbox{and}\ \ p\leq a.
\end{split}\end{equation*}
Then there exist $C, C'>0$ such that for all $f\in W^{1, 2}(\Om)\cap L^{\frac{s}{s_1}}(\Om)$,
\begin{equation*}
\|f\|_{W^{p, 2}(\Om)}\leq C\|\nabla f\|^a_{L^2(\Om)}\|f\|^{1-a}_{L^{\frac{s}{s_1}}(\Om)}+C'\|f\|_{L^{\frac{s}{s_1}}(\Om)}.
\end{equation*}
\end{lemma}

In order to prove Theorem \ref{jbczs}, we need some important properties of the solution to problem \eqref{1} in general bounded domain. Firstly, we give following conclusion on the boundary integration, which will serve Lemma \ref{144444} to get the boundedness of $u$ in $L^\infty([0, T]; W^{1, q}(\Om))$.

\begin{lemma}\label{bjt4}
Assume $\Om$ be a bounded domain with smooth boundary and $q>n>1$. For the solution of problem \eqref{1} there exists positive constant $C$ which is independent of $\epsilon$ and $t$ such that
\begin{equation}\label{ftgj}\begin{split}
\frac{1}{2}\int_{\partial\Om}|\nabla u|^{q-2}\frac{\partial|\nabla u|^2}{\partial\nu}\leq\frac{4(q-2)}{q^2}\int_\Om\left(\nabla(|\nabla u|^{\frac{q}{2}})\right)^2+C\int_\Om|\nabla u|^q.
\end{split}\end{equation}
\end{lemma}

\begin{proof}
Applying Lemmata \ref{bjt}, \ref{bjt2} and \ref{bjt3}, for some positive constants $C$ and $C'$ we can see
\begin{equation*}\begin{split}
\int_{\partial\Om}|\nabla u|^{q-2}\frac{\partial|\nabla u|^2}{\partial\nu}&\leq C\int_{\partial\Om}|\nabla u|^{q}\\
&=C\left\||\nabla u|^{\frac{q}{2}}\right\|^2_{L^2(\partial\Om)}\\
&\leq C\left\||\nabla u|^{\frac{q}{2}}\right\|^2_{W^{\tilde{r}+\frac{1}{2}, 2}(\Om)}\\
&\leq C\left\|\nabla|\nabla u|^{\frac{q}{2}}\right\|^{2a}_{L^2(\Om)}\left\||\nabla u|^{\frac{q}{2}}\right\|^{2(1-a)}_{L^2(\Om)}+C'\left\||\nabla u|^{\frac{q}{2}}\right\|^{2}_{L^2(\Om)},
\end{split}\end{equation*}
where $\tilde{r}\in(0, \frac{1}{2})$ and $a=\tilde{r}+\frac{1}{2}\in(0, 1)$. Then by Young's inequality we know there exists constant $C(q)>0$ such that
\begin{equation*}\begin{split}
C\left\|\nabla|\nabla u|^{\frac{q}{2}}\right\|^{2a}_{L^2(\Om)}\left\||\nabla u|^{\frac{q}{2}}\right\|^{2(1-a)}_{L^2(\Om)}\leq\frac{8(q-2)}{q^2}\left\|\nabla|\nabla u|^{\frac{q}{2}}\right\|^{2}_{L^2(\Om)}+C(q)\left\||\nabla u|^{\frac{q}{2}}\right\|^{2}_{L^2(\Om)},
\end{split}\end{equation*}
this leads to \eqref{ftgj} immediately.
\end{proof}

The following logarithmic Sobolev inequality is useful for us, which comes from \cite{Ogawa2003} and has been used as a tool in \cite{kang2016blowup,tian2018hyperbolic} for related chemotaxis models.

\begin{lemma}
Let $n<q<\infty$ and $f\in W^{1, q}(\Om)$, there is constant $C>0$ depending on $n, q$ and $\Om$ such that
\begin{equation}\label{fyd}
\|f\|_{L^\infty(\Om)}\leq C\left[1+\|f\|_{BMO(\Om)}\left(1+\log^+\|f\|_{W^{1,q}(\Om)}\right)\right],
\end{equation}
where $\log^+g=\log g$ for $g>1$ and $\log^+g=0$ for otherwise.
\end{lemma}

\begin{proof}
When $n=3$ and $f\in W^{s, q}(\Om)$ with $s>\frac{3}{q}$, \eqref{fyd} has been proved in \cite[Lemma 2.8]{Ogawa2003}. In fact, copy the proof of \cite[Lemma 2.8]{Ogawa2003} with very small adjustments we also get \eqref{fyd} for any $n\geq1$, $n<q<\infty$ and $f\in W^{1, q}(\Om)$, so we omit the details here.
\end{proof}

The next lemma guarantees that the classical solution of problem \eqref{1} is bounded locally in $L^\infty([0, T]; W^{1, q}(\Om))$ (cf. \cite[Lemma 6]{kang2016blowup}).

\begin{lemma}\label{144444}
For any classical solution $(u, v, w)$ of problem \eqref{1} with the initial data satisfy \eqref{cszjs} and $q>n\geq1$, If
\begin{itemize}
  \item [(i).] $\Om$ is a bounded domain, $\epsilon\in(0, \epsilon_0)$ for all $\epsilon_0>0$; or
  \item [(ii).] $\Om$ is a bounded  convex domain, $\epsilon>0$,
\end{itemize}
then there exists positive number $C$ such that
\begin{equation}\label{2.5}
\|u\|_{W^{1, q}(\Om)}^q\leq\frac{1}{\sqrt{\left|\left(\|u_{0}\|_{W^{1, q}(\Om)}^{-2q}+1\right)e^{-2Ct}-1\right|}},
\end{equation}
where the constant $C$ depends on $\epsilon_0, r, d_1, d_2, \alpha, \beta, \gamma, \delta, \Om$ for the item (i), while for the item (ii), it depends on $r, d_1, d_2, \alpha, \beta, \gamma, \delta, \Om$.
\end{lemma}

\begin{proof}
We first give an estimate for $\frac{d}{dt}\|u\|_{L^q(\Om)}^q$. Testing the first equality of \eqref{xd1} with $u^{q-1}$ and using integration by part, the second equation of \eqref{xd1} and the nonnegativity of $u, v$ tell us
\begin{equation}\label{lpgj}\begin{split}
\frac{1}{q}\frac{d}{dt}\|u\|_{L^q(\Om)}^q&=-\frac{4\epsilon(q-1)}{q^2}\int_\Om|\nabla u^{\frac{q}{2}}|^{2}+(q-1)\int_\Om u^{q-1}\nabla v\cdot\nabla u\\
&\quad\ +\int_\Om(ru^q-\mu u^{\theta+q-1})\\
&\leq-\frac{(q-1)}{q}\int_\Om u^q\Delta v+r\int_\Om u^q\\
&\leq\frac{\alpha(q-1)}{d_1q}\int_\Om u^q w+r\int_\Om u^q\\
&\leq\left(\frac{\alpha\gamma(q-1)}{d_1q\gamma}\|u\|_{L^\infty(\Om)}+r\right)\|u\|_{L^q(\Om)}^q,
\end{split}\end{equation}
where we also used the elliptic estimate $\|w\|_{L^\infty(\Om)}\leq\frac{\gamma}{\delta}\|u\|_{L^\infty(\Om)}$, which can be found in Lemma \ref{2ly}.

Then we aim to get an estimate for $\frac{d}{dt}\|\nabla u\|_{L^q(\Om)}^q$. By the second equation in \eqref{xd1} we further note
\begin{equation*}\begin{split}
-\nabla\cdot(u\nabla v)&=-\nabla u\cdot\nabla v-u\Delta v\\
&=-\nabla u\cdot\nabla v+\frac{\alpha}{d_1}uw-\frac{\beta}{d_1}uv.
\end{split}\end{equation*}
Then we can rewrite the first equation in \eqref{xd1} as
\begin{equation*}\begin{split}
u_{t}=\epsilon\Delta u-\nabla u\cdot\nabla v+\frac{\alpha}{d_1}uw-\frac{\beta}{d_1}uv+ru-\mu u^\theta.
\end{split}\end{equation*}
Differentiate both sides of above equation for every variable $x_k$, and let $z=u_{x_k}$, here $k=1, 2, \cdots, n$, then
\begin{equation*}\begin{split}
z_{t}=\epsilon\Delta z-\nabla z\cdot\nabla v-\nabla u\cdot\nabla v_{x_k}+\frac{\alpha}{d_1}zw+\frac{\alpha}{d_1}uw_{x_k}-\frac{\beta}{d_1}zv-\frac{\beta}{d_1}uv_{x_k}+rz-\mu\theta u^{\theta-1}z.
\end{split}\end{equation*}
Multiplying above equation by $|\nabla u|^{q-2}z$, adding them for all $k=1, 2, \cdots, n$ and integrating over $\Om$, we obtain
\begin{equation*}\begin{split}
\frac{1}{q}\frac{d}{dt}\|\nabla u\|_{L^q(\Om)}^q&=\epsilon\int_\Om|\nabla u|^{q-2}\nabla u\cdot\nabla\Delta u-\int_\Om|\nabla u|^{q-2}\nabla u\nabla(\nabla u)\nabla v\\
&\quad-\int_\Om|\nabla u|^{q}D^2v+\frac{\alpha}{d_1}\int_\Om|\nabla u|^{q}w+\frac{\alpha}{d_1}\int_\Om|\nabla u|^{q-2}\nabla u\cdot u\nabla w\\
&\quad-\frac{\beta}{d_1}\int_\Om|\nabla u|^{q}v-\frac{\beta}{d_1}\int_\Om|\nabla u|^{q-2}\nabla u\cdot u\nabla v\\
&\quad+r\int_\Om|\nabla u|^{q}-\mu\theta\int_\Om|\nabla u|^{q}u^{\theta-1}.
\end{split}\end{equation*}
We utilize the fact $\Delta|\nabla u|^2=2\nabla u\cdot\nabla\Delta u+2|D^2u|^2$ and Young's inequality as well as the nonnegativity of $u$ and $v$ to see that
\begin{equation}\label{didc}\begin{split}
\frac{1}{q}\frac{d}{dt}\|\nabla u\|_{L^q(\Om)}^q&\leq\epsilon\int_\Om|\nabla u|^{q-2}\left(\frac{1}{2}\Delta|\nabla u|^2-|D^2u|^2\right)-\frac{1}{q}\int_\Om\nabla\left(|\nabla u|^q\right)\cdot\nabla v\\
&\quad+\|D^2v\|_{L^\infty(\Om)}\int_\Om|\nabla u|^{q}+\frac{\alpha}{d_1}\|w\|_{L^\infty(\Om)}\int_\Om|\nabla u|^{q}\\
&\quad+\frac{\alpha(q-1)}{d_1q}\int_\Om|\nabla u|^{q}+\frac{\alpha}{d_1q}\|u\|^q_{L^\infty(\Om)}\int_\Om|\nabla w|^q\\
&\quad+\frac{\beta(q-1)}{d_1q}\int_\Om|\nabla u|^{q}+\frac{\beta}{d_1q}\|u\|^q_{L^\infty(\Om)}\int_\Om|\nabla v|^q+r\int_\Om|\nabla u|^{q}.
\end{split}\end{equation}
It follows from integration by part and the second equation of \eqref{xd1} that
\begin{equation*}\begin{split}
&\epsilon\int_\Om|\nabla u|^{q-2}\left(\frac{1}{2}\Delta|\nabla u|^2-|D^2u|^2\right)-\frac{1}{q}\int_\Om\nabla\left(|\nabla u|^q\right)\cdot\nabla v\\
&\leq\frac{\epsilon}{2}\int_{\partial\Om}|\nabla u|^{q-2}\frac{\partial|\nabla u|^2}{\partial\nu}-\frac{4(q-2)\epsilon}{q^2}\int_\Om\left(\nabla(|\nabla u|^{\frac{q}{2}})\right)^2+\frac{1}{q}\int_\Om|\nabla u|^q\Delta v\\
&\leq\frac{\epsilon}{2}\int_{\partial\Om}|\nabla u|^{q-2}\frac{\partial|\nabla u|^2}{\partial\nu}-\frac{4(q-2)\epsilon}{q^2}\int_\Om\left(\nabla(|\nabla u|^{\frac{q}{2}})\right)^2+\frac{\beta}{d_1q}\int_\Om|\nabla u|^q v.
\end{split}\end{equation*}
When $n=1$, we see $\int_{\partial\Om}|\nabla u|^{q-2}\frac{\partial|\nabla u|^2}{\partial\nu}=0$ due to the Neumann boundary condition. Hence, when $n\geq1$, for the item (i), i.e., $\epsilon\in(0,\ \epsilon_0)$, \eqref{ftgj} tell us
\begin{equation}\label{cslm}\begin{split}
&\epsilon\int_\Om|\nabla u|^{q-2}\left(\frac{1}{2}\Delta|\nabla u|^2-|D^2u|^2\right)-\frac{1}{q}\int_\Om\nabla\left(|\nabla u|^q\right)\cdot\nabla v\\
&\leq\left(C+\frac{\alpha\lambda}{d_1q\gamma}\|u\|_{L^\infty(\Om)}\right)\int_\Om|\nabla u|^q,
\end{split}\end{equation}
where we also used the elliptic estimate twice: $\|v\|_{L^\infty(\Om)}\leq\frac{\alpha}{\beta}\|w\|_{L^\infty(\Om)}\leq\frac{\alpha\gamma}{\beta\delta}\|u\|_{L^\infty(\Om)}$, which can be found in Lemma \ref{2ly}, and $C=C(\epsilon_0, n, \Om)>0$. While for the item (ii), we know $\left.\frac{\partial|\nabla u|^2}{\partial\nu}\right|_{\partial\Om}\leq0$ due to the convexity of $\Om$, so \eqref{cslm} still holds with $C=0$. Thus, for both items (i) and (ii), \eqref{didc} implies
\begin{equation}\label{didc2}\begin{split}
\frac{1}{q}\frac{d}{dt}\|\nabla u\|_{L^q(\Om)}^q&\leq\Psi_0\int_\Om|\nabla u|^{q}+\frac{\alpha+\beta}{d_1q}\|u\|^q_{L^\infty(\Om)}\left(\int_\Om|\nabla w|^q+\int_\Om|\nabla v|^q\right),
\end{split}\end{equation}
where
\begin{equation*}
\Psi_0:=C+\frac{\alpha\gamma(q+1)}{d_1q\delta}\|u\|_{L^\infty(\Om)}+\|D^2v\|_{L^\infty(\Om)},
\end{equation*}
and $C=C(\epsilon_0, \alpha, \beta, d_1, r, q, \Om)>0$ for item (i), while for item (ii), the constant $C$ only depends on $\alpha, \beta, d_1, r, q, \Om$. We combine \eqref{lpgj} and \eqref{didc2} to obtain
\begin{equation}\label{didc9}\begin{split}
\frac{d}{dt}\|u\|_{W^{1, q}(\Om)}^q&\leq\Psi_1\|u\|_{W^{1, q}(\Om)}^q+\frac{\alpha+\beta}{d_1}\|u\|^q_{L^\infty(\Om)}\left(\int_\Om|\nabla w|^q+\int_\Om|\nabla v|^q\right),
\end{split}\end{equation}
where
\begin{equation*}
\Psi_1:=C+\frac{2q\alpha\gamma}{d_1\delta}\|u\|_{L^\infty(\Om)}+ q\|D^2v\|_{L^\infty(\Om)}.
\end{equation*}

With the help of the second equation of \eqref{1} and the elliptic estimate listed in Lemma \ref{2ly} we note there exists positive constant $C$ such that
\begin{equation*}
\|\nabla v\|_{L^q(\Om)}^q\leq \|v\|_{W^{2, q}(\Om)}^q\leq C\|w\|_{L^q(\Om)}^q\leq C\|w\|_{W^{1, q}(\Om)}^q.
\end{equation*}
Similarly, by the third equation of \eqref{1} and the elliptic estimate we further get
\begin{equation*}
\|\nabla w\|_{L^q(\Om)}^q\leq \|w\|_{W^{1, q}(\Om)}^q\leq \|w\|_{W^{2, q}(\Om)}^q\leq C\|u\|_{L^q(\Om)}^q\leq C\|u\|_{W^{1, q}(\Om)}^q.
\end{equation*}
Insert above two relations into \eqref{didc9}, then
\begin{equation}\label{didc10}\begin{split}
\frac{d}{dt}\|u\|_{W^{1, q}(\Om)}^q&\leq\Psi_2\|u\|_{W^{1, q}(\Om)}^q,
\end{split}\end{equation}
where
\begin{equation}\label{sitaa3}
\Psi_2:=C\left(1+\|u\|^q_{L^\infty(\Om)}+\|D^2v\|_{L^\infty(\Om)}\right).
\end{equation}

Finally, we estimate $\Psi_2$. Since $q>n$, by means of Sobolev imbedding theorem we know there exists positive constant such that  $\|u\|^q_{L^\infty(\Om)}\leq C\|u\|^q_{W^{1, q}(\Om)}$. Then, we only need to estimate $\|D^2v\|_{L^\infty(\Om)}$. By BMO and $L^p$ estimates of elliptic equations \cite[Page 259]{wuyin}, for some positive constants there hold
\begin{equation*}
\|D^2v\|_{BMO(\Om)}\leq C\|u\|_{L^\infty(\Om)}\ \mbox{and}\ \|D^2v\|_{W^{1,q}(\Om)}\leq C\|u\|_{W^{1,q}(\Om)}.
\end{equation*}
Applying \eqref{fyd} to $D^2v$ and using Sobolev imbedding theorem again, then from the fact that $\xi\log\xi\leq\xi^2$ for all $\xi>0$ we observe
\begin{equation*}\begin{split}
\|D^2v\|_{L^\infty(\Om)}&\leq C\left[1+\|D^2v\|_{BMO(\Om)}\left(1+\log^+\|D^2v\|_{W^{1,q}(\Om)}\right)\right]\\
&\leq C\left[1+\|u\|_{L^\infty(\Om)}\left(1+\log^+\|u\|_{W^{1,q}(\Om)}\right)\right]\\
&\leq C\left[1+\|u\|_{W^{1,q}(\Om)}\left(1+\log\|u\|_{W^{1,q}(\Om)}\right)\right]\\
&\leq C\left(1+\|u\|^2_{W^{1,q}(\Om)}\right).
\end{split}\end{equation*}
Then \eqref{sitaa3} and Young's inequality lead to $\Psi_2\leq C\left(1+\|u\|^{2q}_{W^{1,q}(\Om)}\right)$. We can infer from \eqref{didc10} that
\begin{equation*}
\frac{d}{dt}\|u\|_{W^{1, q}(\Om)}^q\leq C\left(\|u\|^q_{W^{1,q}(\Om)}+\|u\|^{3q}_{W^{1,q}(\Om)}\right),
\end{equation*}
which is a Bernoulli-type inequality, solving this inequality we can see \eqref{2.5} holds so our proof is complete.
\end{proof}

\section{Proofs of Theorems \ref{jbczs} and \ref{d1q}}
\subsection{Local well-posed for problem \eqref{2}}

Without any restrictions on the space dimension, this section is devoted to find out that whether the transient growth phenomena will hold to problem \eqref{1}. Especial speaking, we consider problem \eqref{1} in non-radial symmetry case, and then obtain clear conditions such that the phenomena shall occur in the sense that the solution surpasses any given threshold. As mentioned above, the main tool in this section are the viscosity vanishing method and some important estimates, then an auxiliary problem with the corresponding hyperbolic-elliptic-elliptic model, i.e., problem \eqref{2}, is needed for us. So, in order to avoid any confusions we shall denote the solution of problem \eqref{1} with any $\epsilon>0$ by $(u_\epsilon, v_\epsilon, w_\epsilon)$ in this section and rewrite problem \eqref{1} as
\begin{equation}\label{xd1}
\left\{
\begin{array}{ll}
u_{\epsilon t}=\epsilon\Delta u_\epsilon-\nabla\cdot(u_\epsilon\nabla v_\epsilon)+ru_\epsilon-\mu u_\epsilon^\theta,\quad &x\in \Omega,\quad t>0,\\
0=d_1\Delta v_\epsilon-\beta v_\epsilon+\alpha w_\epsilon,\quad &x\in \Omega,\quad t>0,\\
0=d_2\Delta w_\epsilon-\delta w_\epsilon+\gamma u_\epsilon,\quad &x\in\Omega,\quad t>0,\\
\ds\frac{\partial u_\epsilon}{\partial\nu }=\frac{\partial v_\epsilon}{\partial \nu}=\frac{\partial w_\epsilon}{\partial\nu }=0,\quad &x\in\partial\Omega,\quad t>0,\\
u_\epsilon(x, 0)=u_{0\epsilon}\, \,\quad &x\in\Omega.
\end{array}
\right.
\end{equation}

In a nutshell, we will prove that the strong $W^{1,q}$-solution to problem \eqref{2} is a approximating solution of problem \eqref{xd1}, i.e., $(u_\epsilon, v_\epsilon, w_\epsilon)\rightarrow(u, v, w)$ as $\epsilon\rightarrow 0$. The main strategies come from \cite{Chemotaxis,kang2016blowup,Winkler2014How}, wherein the corresponding hyperbolic-elliptic type system related to the classical Keller-Segel growth model was studied. Kang and Stevens \cite{kang2016blowup} extended the radially symmetric setting as in \cite{Chemotaxis,Winkler2014How} to a bounded but convex domain via a logarithmic Sobolev inequality in terms of BMO norm and Sobolev norm \cite{Kozono2000}. The novelty of this part is that we extend the hyperbolic-elliptic model to problem \eqref{2} which including indirect signal production mechanism, moreover, we remove the convexity hypothesis in \cite{kang2016blowup}.

We first claim that the $W^{1,q}$-solution of problem \eqref{2} is uniquely determined.
\begin{lemma}\label{wyi}
Problem \eqref{2} possesses at most one strong $W^{1,q}$-solution.
\end{lemma}

\begin{proof}
The uniqueness can be proved by a contradiction argument. However, due to the lack of the regularity of $u$ to problem \eqref{2}, we shall choose suitable test function as in \cite[Lemma 4.2]{Chemotaxis}, \cite[Proposition 2.1]{tian2018hyperbolic} and \cite[Lemma 4.2]{Winkler2014How} and then use Gronwall's lemma to obtain the conclusion. For convenience of reference, we give some main points and necessary adjustments here.

Assume $(u, v, w)$ and $(\bar u, \bar v, \bar w)$ were two strong $W^{1,q}$-solutions of problem \eqref{2}, then from the Definition \ref{21} we know $(U, V, W):=(u-\bar u, v-\bar v, w-\bar w)$ satisfy
\begin{equation}\begin{split}\label{14}
&-\int_0^T\int_\Om U\varphi_t\\
&=\int_0^T\int_\Om\left[\left( u\nabla v-\bar u\nabla\bar v\right)\nabla\varphi+rU\varphi-\mu U\left(u^{\theta-1}+u^{\theta-2}\bar u+\cdots+ u\bar{u}^{\theta-2}+\bar u^{\theta-1}\right)\varphi\right]\\
&=\int_0^T\int_\Om\bigg[-\nabla u\nabla v-u\Delta v+\nabla\bar u\nabla\bar v+\bar u\Delta\bar v\\
&\quad+U\left(r-\mu u^{\theta-1}-\mu u^{\theta-2}\bar u-\cdots-\mu u\bar{u}^{\theta-2}-\mu \bar u^{\theta-1}\right)\bigg]\varphi.
\end{split}\end{equation}
for all $\varphi\in C_0^\infty(\bar\Om\times[0,T))$, and
\begin{equation}\label{15}\begin{split}
&0=d_1\Delta V-\beta V+\alpha W,\\
&0=d_2\Delta W-\delta W+\gamma U.
\end{split}\end{equation}
Let $T_0\in(0, T)$, for $t_0\in (0, T_0)$ and $\delta\in(0, \frac{T-t_0}{2})$ we define $\chi_\delta(t)$ by
\begin{eqnarray*}
\chi_\delta(t)=
\left\{
\begin{array}{ll}
1,\quad &t<t_0,\\
\frac{t_0-t+\delta}{\delta},\quad & t\in[t_0, t_0+\delta],\\
0,\quad & t>t_0+\delta.
\end{array}
\right.
\end{eqnarray*}
For such $\delta$ and $h\in(0, \frac{T_0-t_0}{2})$ and $\eta\in(0, 1)$ we further define
\begin{equation*}
\varphi(x,t)=\frac{\chi_\delta(t)}{h}\int_{t}^{t+h}U(U^2+\eta)^{\frac{q}{2}-1},\quad (x,t)\in\Om\times(0,T).
\end{equation*}
Then we can take $\varphi(x,t)$ as a test function in \eqref{14} and use the definition of $\varphi(x,t)$ to obtain
\begin{equation*}\begin{split}
&\frac{1}{\delta}\int_{t_0}^{t_0+\delta}\int_\Om\frac{U}{h}\int_{t}^{t+h}U(U^2+\eta)^{\frac{q}{2}-1}\\
&\quad-\int_0^T\int_\Om\chi_\delta(t)U\frac{U(x,t+h)(U^2(x,t+h)+\eta)^{\frac{q}{2}-1}-U(U^2+\eta)^{\frac{q}{2}-1}}{h}\\
&=\int_0^T\int_\Om\chi_\delta(t)\frac{-\nabla u\nabla v-u\Delta v+\nabla\bar u\nabla\bar v+\bar u\Delta\bar v}{h}\int_{t}^{t+h}U(U^2+\eta)^{\frac{q}{2}-1}\\
&\quad +\int_0^T\int_\Om\chi_\delta(t)\frac{U\left(r-\mu u^{\theta-1}-\mu u^{\theta-2}\bar u-\cdots-\mu u\bar{u}^{\theta-2}-\mu \bar u^{\theta-1}\right)}{h}\int_{t}^{t+h}U(U^2+\eta)^{\frac{q}{2}-1}\\
&=\int_0^T\int_\Om\chi_\delta(t)\frac{-\nabla U\nabla v-U\Delta v-\nabla\bar u\nabla V-\bar u\Delta V}{h}\int_{t}^{t+h}U(U^2+\eta)^{\frac{q}{2}-1}\\
&\quad +\int_0^T\int_\Om\chi_\delta(t)\frac{U\left(r-\mu u^{\theta-1}-\mu u^{\theta-2}\bar u-\cdots-\mu u\bar{u}^{\theta-2}-\mu \bar u^{\theta-1}\right)}{h}\int_{t}^{t+h}U(U^2+\eta)^{\frac{q}{2}-1}.
\end{split}\end{equation*}
Using Lebesgue's theorem and the fact that $\bar u, \nabla v, \Delta v, U, u$ are bounded and $\nabla U$ is uniformly bounded in $L^q(\Omega)$, then with the help of the same argument in the proof of \cite[Lemma 4.2]{Chemotaxis} we can take $\delta\rightarrow0$ and $\eta\rightarrow0$ to see
\begin{equation*}\begin{split}
&\int_\Om\frac{U(x,t_0)}{h}\int_{t_0}^{t_0+h}U|U|^{q-2}+\frac{q-1}{hq}\int_0^{h}\int_\Om |U|^{q}-\frac{q-1}{hq}\int_{t_0}^{t_0+h}\int_\Om |U|^q\\
&\leq\int_0^{t_0}\int_\Om\frac{-\nabla U\nabla v-U\Delta v-\nabla\bar u\nabla V-\bar u\Delta V}{h}\int_{t}^{t+h}U|U|^{q-2}\\
&\quad+\int_0^{t_0}\int_\Om\frac{U\left(r-\mu u^{\theta-1}-\mu u^{\theta-2}\bar u-\cdots-\mu u\bar{u}^{\theta-2}-\mu \bar u^{\theta-1}\right)}{h}\int_{t}^{t+h}U|U|^{q-2}.
\end{split}\end{equation*}
As $U$ is continuous on $\bar\Omega\times [0, T_0]$ and $U(\cdot, 0)=0$ in $\Omega$, then let $h\rightarrow 0$ in above inequality yields
\begin{equation}\label{171}\begin{split}
\frac{1}{q}\int_\Om |U(x,t_0)|^q&\leq\int_0^{t_0}\int_\Om\left(-\nabla U\nabla v-U\Delta v-\nabla\bar u\nabla V-\bar u\Delta V\right)\cdot U|U|^{q-2}\\
&\quad +\int_0^{t_0}\int_\Om\left(r-\mu u^{\theta-1}-\mu u^{\theta-2}\bar u-\cdots-\mu u\bar{u}^{\theta-2}-\mu \bar u^{\theta-1}\right)\cdot |U|^{q}.
\end{split}\end{equation}

Finally, we will estimate the integrations on the right hand side to obtain an expression that allows to conclude $U=0$ by means of Gronwall's inequality. In this end, we define some constants for convenience'sake: $c_1=\|u\|_{L^\infty(\Om\times(0,T_0))}, \ c_2=\|v\|_{L^\infty(\Om\times(0,T_0))},\ c_3=\|\bar u\|_{L^\infty(\Om\times(0,T_0))}$ and $c_4=\|\bar u\|_{L^\infty((0,T_0); W^{1,q}(\Om))}$. We first integrate by parts and use the second equation in \eqref{2} to conclude
\begin{equation*}\begin{split}
-\int_0^{t_0}\int_\Om \nabla U\nabla vU|U|^{q-2}&=\frac{1}{q}\int_0^{t_0}\int_\Om |U|^q\Delta v\\
&=\frac{1}{q}\int_0^{t_0}\int_\Om |U|^q\left(\frac{\beta v-\alpha w}{d_1}\right)\\
&\leq\frac{\beta}{qd_1}c_2\int_0^{t_0}\int_\Om |U|^q.
\end{split}\end{equation*}
Similarly, for the next term in the right hand side of \eqref{171} we have
\begin{equation*}\label{17}\begin{split}
-\int_0^{t_0}\int_\Om \Delta v|U|^{q}\leq\frac{\alpha c_1}{d_1}\int_0^{t_0}\int_\Om |U|^q.
\end{split}\end{equation*}
Applying Lemma \ref{2ly} twice to the equations in \eqref{15} and using by H\"{o}lder's inequality we know there is a constant $C>0$ such that
\begin{equation*}\label{17}\begin{split}
-\int_0^{t_0}\int_\Om \nabla\bar u\nabla VU|U|^{q-2}&\leq c_4\int_0^{t_0}\|\nabla V\|_{L^q(\Om)}\left\|U|U|^{q-2}\right\|_{L^{\frac{q}{q-1}}(\Om)}\\
&\leq c_4C\int_0^{t_0}\|U\|_{L^q(\Om)}\|U\|^{q-1}_{L^q(\Om)}.
\end{split}\end{equation*}
Similarly,
\begin{equation*}\begin{split}
-\int_0^{t_0}\int_\Om \bar u\Delta VU|U|^{q-2}&= c_3\int_0^{t_0}\|\Delta V\|_{L^q(\Om)}\left\|U|U|^{q-2}\right\|_{L^{\frac{q}{q-1}}(\Om)}\\
&\leq c_3C\int_0^{t_0}\|U\|_{L^q(\Om)}\|U\|^{q-1}_{L^q(\Om)}.
\end{split}\end{equation*}
Using of the nonnegativity of $u$ and $\bar u$ we can see
\begin{equation*}\label{17}\begin{split}
\int_0^{t_0}\int_\Om\left(r-\mu u^{\theta-1}-\mu u^{\theta-2}\bar u-\cdots-\mu u\bar{u}^{\theta-2}-\mu \bar u^{\theta-1}\right)\cdot |U|^{q}\leq r\int_0^{t_0}\int_\Om |U|^q.
\end{split}\end{equation*}
Gathering all these estimates together we see for $t_0\in (0, T_0)$ and constant $C>0$
\begin{equation*}
\int_\Om |U(x,t_0)|^q\leq C\int_0^{t_0}\int_\Om |U|^q,
\end{equation*}
hence it follows from Gronwall's lemma that $U=0$, then $W=0$, this further implies $V=0$, so our proof is complete.
\end{proof}

We next give a general convergence result of solutions between $(u_\epsilon, v_\epsilon, w_\epsilon)$ and $(u, v, w)$.

\begin{lemma}\label{25}
Let the initial value of problem \eqref{2} satisfy $u_0\in W^{1, q}(\Om), q>n, T>0$. For some zero sequence $(\epsilon)_{j\in\mathbb{N}}\subset(0, 1)$ such that $\epsilon\rightarrow 0$ as $j\rightarrow\infty$, if for any $\epsilon_j\in(\epsilon_j)_{j\in\mathbb{N}}$ negative initial data $\ u_{0\epsilon_{j}}$ of problem \eqref{xd1} with $\|u_{0\epsilon_{j}}-u_0\|_{W^{1,q}(\Om)}<\epsilon_{j}$ the corresponding solution $u_{d_{j}}$ satisfies
\begin{equation*}\label{18}
\|u_{\epsilon_{j}}\|_{L^\infty(\Om)}\leq M,\quad (x, t)\in\Om\times(0, T),
\end{equation*}
then there exists a strong $W^{1,q}$-solution $(u, v, w)$ of problem \eqref{2} in $\Om\times(0, T)$ such that
\begin{equation}\label{zg}\begin{split}
&u_{\epsilon_{j}}\rightarrow u\ \ \mbox{ in }\ C(\bar\Om\times[0,T]),\\
&u_{\epsilon_{j}}\rightharpoonup^{*} u\ \mbox{ in }\ L^\infty((0,T); W^{1,q}(\Omega))\\
&v_{\epsilon_{j}}\rightarrow v\ \mbox{ in }\ C^{2,0}(\bar\Om\times[0,T]),\\
&w_{\epsilon_{j}}\rightarrow w\ \mbox{ in }\ C^{2,0}(\bar\Om\times[0,T]),
\end{split}\end{equation}
as $\epsilon_{j}\rightarrow0$, i.e., $j\rightarrow\infty$.
\end{lemma}

\begin{proof}
By Lemma \ref{144444} we know there is a $T>0$ such that for all $q>n$,
\begin{equation}\label{uyj}
u_\epsilon(x,t)\in L^\infty([0, T]; W^{1, q}(\Om)).
\end{equation}
Then, we can utilize elliptic estimates to the two elliptic equations in \eqref{xd1} to further get $w_\epsilon\in L^\infty([0, T]; W^{3, q}(\Om))$ and $v_\epsilon\in L^\infty([0, T]; W^{5, q}(\Om))$. Since the positive coefficient $C$ \eqref{2.5} is independent of $\epsilon$ and $t$, the boundedness we obtained is uniform in time and $\epsilon$. Furthermore, due to the boundedness of $\epsilon$, as noted in Remark \ref{r3} already, we can get a suitable strong regularity of the time derivative of solution to problem \eqref{xd1} by the same ways as in \cite[Lemma 3.4]{Winkler2014How} and \cite[Lemma 3.10]{Chemotaxis}, that is, for $p>1$,
\begin{equation}\label{utyj}
u_{\epsilon t}(x,t)\in L^p\left([0, T]; (W^{1, \frac{q}{q-1}}(\Om))^*\right).
\end{equation}
Hence, the rest of the proof can be done by a standard compactness argument, cf. \cite[Lemma 4.4]{Chemotaxis} and \cite[Lemma 4.3]{Winkler2014How}. We include a short proof here for the sake of completeness and readers' convenience.

Make use of \eqref{uyj}, \eqref{utyj} and a variant of the classical Aubin-Lions lemma (see \cite[Lemma 4.4]{Winkler2014How}), we know for any subsequence of $(\epsilon_j)_j$ one can pick a further subsequence thereof such that
\begin{equation}\label{sxk}\begin{split}
&u_{\epsilon_{j_i}}\rightarrow u\ \ \mbox{ in }\ C([0,T]; C^\iota(\Om)),\\
&u_{\epsilon_{j_i}}\rightharpoonup^{*} u\ \mbox{ in }\ L^\infty((0,T); W^{1,q}(\Omega))
\end{split}\end{equation}
as $i\rightarrow\infty$. Moreover, from the propagation of the Cauchy-property and Lemma \ref{2ly} we know that there exists positive constant $C$ such that
\begin{equation*}\begin{split}
\left\|w_{\epsilon_{j_i}}-w_{\epsilon_{j_k}}\right\|_{C^{2,0}(\Om\times[0,T])}&\leq\left\|w_{\epsilon_{j_i}}-w_{\epsilon_{j_k}}\right\|_{C^{2+\iota,0}
(\Om\times[0,T])}\\
&\leq C\left\|u_{\epsilon_{j_i}}-u_{\epsilon_{j_k}}\right\|_{C^{\iota,0}(\Om\times[0,T])}
\end{split}\end{equation*}
for any $i, k\in\mathbb{N}$. Thus, \eqref{sxk} gives $w_{\epsilon_{j_i}}\rightarrow w$ in $C^{2,0}(\Om\times[0,T])$. Similarly, we can also get the convergence for element $v$ in $C^{2,0}(\Om\times[0,T])$, thus \eqref{zg} holds. Finally, testing the first equation of problem \eqref{xd1} by an arbitrary $\varphi\in C_0^\infty([0, T); \bar\Omega)$,
\begin{equation*}\begin{split}
&-\int_0^T\int_\Om u_\epsilon\varphi_t-\int_\Om u_{0\epsilon}\varphi(\cdot, 0)\\
&=-\epsilon\int_0^T\int_\Om\nabla u_\epsilon\nabla\varphi+\int_0^T\int_\Om u_\epsilon\nabla v_\epsilon\nabla\varphi+\int_0^T\int_\Omega (ru_\epsilon-\mu u_\epsilon^\theta)\varphi.
\end{split}\end{equation*}
Taking $\epsilon=\epsilon_{j_i}\rightarrow0$ in above equation, \eqref{zg} implies \eqref{13} immediately. Moreover, due to the uniqueness of solution as obtained in Lemma \ref{wyi}, we know the whole sequence converges to the solution $(u, v, w)$ of problem \eqref{2}.
\end{proof}

Now we show the local existence of the strong $W^{1,p}$-solution of problem \eqref{2} under small initial data.

\begin{lemma}\label{3.2}
Let $D>0$ be any constant, $(\epsilon_j)_{j\in\mathbb{N}}$ be the zero sequence given by Lemma \ref{25} as well as the initial value of problem \eqref{2} satisfies $\|u_0\|_{W^{1,q}(\Om)}\leq D$, then there is number $T(D)>0$ such that problem \eqref{2} admits a strong $W^{1, p}$-solution $(u, v, w)$ in $\Omega\times(0, T(D))$. Furthermore, for some $\epsilon_j\in(\epsilon_j)_{j\in\mathbb{N}}$ if the initial value $u_{0\epsilon_{j}}$ of problem \eqref{xd1} satisfies
\begin{equation}\label{ng2}
\|u_{0\epsilon_{j}}-u_0\|_{W^{1,q}(\Om)}<\epsilon_{j},
\end{equation}
then $(u, v, w)$ can be approximated by solution $(u_{\epsilon_j}, v_{\epsilon_j}, w_{\epsilon_j})$ of problem \eqref{xd1} in the following sense:
\begin{equation*}\label{slu}\begin{split}
&u_{\epsilon_j}\rightarrow u\  \ \mbox{ in }\ C(\bar\Om\times[0, T(D)]),\\
&u_{\epsilon_j}\rightharpoonup^{*} u\ \ \mbox{ in }\ L^\infty((0, T(D)); W^{1,q}(\Omega))\\
&v_{\epsilon_j}\rightarrow v\ \mbox{ in }\ C^{2, 0}(\bar\Om\times[0, T(D)]),\\
&w_{\epsilon_j}\rightarrow w\ \mbox{ in }\ C^{2,0}(\bar\Om\times[0, T(D)])
\end{split}\end{equation*}
as $\epsilon_j\rightarrow0$. Moreover, element $u$ of problem \eqref{2} satisfies the estimate \eqref{2.5}.
\end{lemma}

\begin{proof}
We can infer from Lemma \ref{144444} that solutions $u_{\epsilon_j}$ to problem \eqref{xd1} exists bounded on $(0, T(D))$, then Lemma \ref{25} tells us there is a strong $W^{1,q}$-solution of problem \eqref{2} with the claimed approximation properties, so we only need to prove estimate \eqref{2.5} for solution of problem \eqref{2}.

According to \eqref{2.5} we know the solution component $u_{\epsilon_j}$ of problem \eqref{xd1} satisfies
\begin{equation*}
\|u_{\epsilon_j}\|_{W^{1, q}(\Om)}^q\leq\frac{1}{\sqrt{\left|\left(\|u_{0\epsilon_j}\|_{W^{1, q}(\Om)}^{-2q}+1\right)e^{-2Ct}-1\right|}}.
\end{equation*}
It follows from \eqref{ng2} that the convergence of the right hand side is obvious as $\epsilon_j\rightarrow0$, moreover, Lemma \ref{144444} implies the boundedness of $(u_{\epsilon_j})_j$ and $(\nabla u_{\epsilon_j})_j$ in $L^q(\Om)$, so by $L^q$-weak convergence along a subsequence and the weak lower semicontinuity of the norm we can obtain
\begin{equation*}\begin{split}
\|u(\cdot, t)\|_{W^{1, q}(\Om)}^q&\leq\liminf_{k\rightarrow\infty}\left(\|u_{\epsilon_{j_k}}(\cdot, t)\|_{W^{1, q}(\Om)}^q\right)\\
&\leq\frac{1}{\sqrt{\left|\left(\|u_{0}\|_{W^{1, q}(\Om)}^{-2q}+1\right)e^{-2Ct}-1\right|}},
\end{split}\end{equation*}
then our proof is complete.
\end{proof}

Solutions of problem \eqref{2} constructed up to now may only exist with small initial data on time interval $(0, T(D))$. We next give the proof of Theorem \ref{jbczs} which suggests that the solution problem \eqref{2} can exist on a maximal time interval and is locally well-posed (cf. \cite[Theorem 1.2]{Winkler2014How}).

\begin{proof}[Proof of Theorem \ref{jbczs}]
With the help of Lemma \ref{wyi}, we can apply Lemma \ref{3.2} with $D:=\|u_0\|_{W^{1,q}(\Om)}$ to gain $T>0$ and a unique strong $W^{1, q}$-solution of problem \eqref{2} in $\Om\times(0, T)$ fulfilling
\begin{equation}\label{4.47}
\begin{split}
\|u(\cdot, t)\|_{W^{1, q}(\Om)}^q\leq\frac{1}{\sqrt{\left|\left(\|u_{0}\|_{W^{1, q}(\Om)}^{-2q}+1\right)e^{-2Ct}-1\right|}},
\end{split}
\end{equation}
for almost every $t\in(0, T)$. Accordingly, the set
\begin{equation*}
\begin{split}
S:&=\left\{\tilde{T}>0\ |\ \exists\ \mbox{strong}\ W^{1, q}\mbox{-solution of problem \eqref{2} in}\ \Om\times(0, T)\right.\\
 &\quad\quad\quad\quad\quad\ \left.\mbox{that satisfies \eqref{4.47} for a.e.}\ t\in(0, \tilde{T})\right\}
\end{split}
\end{equation*}
is not empty and $T^*:=\sup S\leq\infty$ is well-defined. The $L^1$-norm boundedness of $u$ to problem \eqref{2} as given by \eqref{hy} can be got by the ODE comparison and a suitable test function as in \cite[Lemma 4.1]{Winkler2014How}, \cite[Lemma 4.7]{Chemotaxis} and \cite[Lemma 4.1]{xu2020carrying}, we omit the detail here. We only have to show the extensibility criterion
\begin{equation*}
\mbox{either}\ \ T^*=+\infty \ \ \mbox{or}\ \ \limsup_{t\rightarrow T^*} \|u(\cdot, t)\|_{L^\infty(\Om)}=+\infty.
\end{equation*}

Arguing with contradiction, we assume
\begin{equation*}
T^*<+\infty\ \ \mbox{and}\ \ \limsup_{t\rightarrow T^*} \|u(\cdot, t)\|_{L^\infty(\Om)}<+\infty,
\end{equation*}
then for all $(x, t)\in\Om\times(0, T^*)$ there exists $M>0$ such that
\begin{equation}\label{yjm}
u(x, t)\leq M.
\end{equation}
Let $N\subset(0, T_{\max})$ be a set of measure zero, as provided by the definition of $S$, such that \eqref{4.47} holds for all $t\in(0, T_{\max})\backslash N$, then \eqref{4.47} gives us
\begin{equation*}
\|u(\cdot, t_0)\|_{W^{1, q}(\Om)}\leq D_1,\quad \forall\ t\in(0, T_{\max})\backslash N
\end{equation*}
with some $D_1>0$. Now by \eqref{yjm} and a second application of Lemma \ref{25} we know for each $t_0\in(0, T_{\max})\backslash N$ there exists $T(D_1)>0$ such that problem \eqref{2} with initial value $u(x, t_0)$ admits a strong $W^{1, q}$-solution $\left(\hat{u}, \hat{v}, \hat{w}\right)$ in $\Om\times(0, T(D_1))$ satisfying
\begin{equation}\label{27}
\begin{split}
\|\hat{u}(\cdot, t)\|_{W^{1, q}(\Om)}^q\leq\frac{1}{\sqrt{\left(\|u(t_0)\|_{W^{1, q}(\Om)}^{-2q}+1\right)e^{-2Ct}-1}},
\end{split}
\end{equation}
for a.e. $t\in(0, T(D_1))$. Thus, choosing any $t_0\in(0, T_{\max})\backslash N$ such that $t_0>T_{\max}-\frac{T(D_1)}{2}$ here, we would infer that
\begin{eqnarray*}
\tilde{u}(\cdot, t)=\left\{
\begin{array}{ll}
u(\cdot, t),\ \  &t\in(0, t_0],\\
\hat{u}(\cdot, t-t_0),\ \ & t\in(t_0, t_0+T(D_1))
\end{array}
\right.
\end{eqnarray*}
would define a strong $W^{1, q}$-solution of problem \eqref{2} in $\Om\times(0, t_0+T(D_1))$, which clearly would satisfy \eqref{4.47} for a.e. $t<t_0$. For $t>t_0$, \eqref{27} gives us
\begin{equation}\label{ncs}
\begin{split}
\|\tilde{u}(\cdot, t)\|_{W^{1, q}(\Om)}^q&\leq\frac{1}{\sqrt{\left(\|u(t_0)\|_{W^{1, q}(\Om)}^{-2q}+1\right)e^{-2C(t-t_0)}-1}}.
\end{split}
\end{equation}
It is easy to see from \eqref{4.47} that
\begin{equation*}\begin{split}
\left(\|u(t_0)\|_{W^{1, q}(\Om)}^{-2q}+1\right)e^{-2C(t-t_0)}&\geq\left(\|u_0\|_{W^{1, q}(\Om)}^{-2q}+1\right)e^{-2Ct_0}e^{-2C(t-t_0)}\\
&=\left(\|u_0\|_{W^{1, q}(\Om)}^{-2q}+1\right)e^{-2Ct},
\end{split}\end{equation*}
then \eqref{ncs} becomes
\begin{equation*}
\begin{split}
\|\tilde{u}(\cdot, t)\|_{W^{1, q}(\Om)}^q&\leq\frac{1}{\sqrt{\left(\|u_0\|_{W^{1, q}(\Om)}^{-2q}+1\right)e^{-2Ct}-1}}.
\end{split}
\end{equation*}
Namely, $\tilde{u}$ satisfy \eqref{4.47} for a.e. $t\in(0, t_0+T(D_1))$, this contradicts the definition of $T^*$, and then our proof is finished.
\end{proof}

\subsection{Blow-up of problem \eqref{2}}

The following technical tool is important for us, which comes from \cite[Lemma 4.9]{Winkler2014How}.

\begin{lemma}\label{31}
Let $a>0, b\geq0, d>0, \kappa>1$ be such that
\begin{equation*}
  a>(\frac{2b}{d})^{\frac{1}{\kappa}},
\end{equation*}
Then if for some $T>0$ the function $y\in C([0, T ))$ is nonnegative and satisfies
\begin{equation*}
  y(t)\geq a-bt+d\int_0^ty^\kappa(s)
\end{equation*}
for all $t\in (0, T)$, we necessarily have
\begin{equation*}
  T\leq\frac{2}{(\kappa-1)a^{\kappa-1}d}.
\end{equation*}
\end{lemma}

\begin{proof}[Proof of Theorem \ref{33}]
Let $(u, v, w)$ be the strong $W^{1,q}$-solution of problem \eqref{2} in $\Om\times(0, T^*)$ with the initial data satisfies \eqref{theta} and $\theta\in(1, 2]$, where $\Om$ is a bounded convex domain. For convenience, we chose the constants $c_1, c_2$ in \eqref{theta} as follow
\begin{equation*}
c_1:=\max\left\{|\Om|^{\frac{1}{\theta+1}}+\left(\frac{24d_1\theta\delta|\Om|^{\frac{1}{\theta}}C}{\alpha\gamma(\theta-1)}\right)^\theta,\ \ \ \left(\frac{12d_1\delta C_3|\Om|^{\frac{1}{\theta}}}{\alpha\gamma}\right)^{\frac{\theta}{\theta+1}} \right\},
\end{equation*}
and
\begin{equation*}
c_2:=\frac{(\theta-1)C_3\pi}{4C},
\end{equation*}
where $C_3, C$ are two positive constants which is independent of $\epsilon, u_0$ and will be introduced later, then by \eqref{theta} we can see
\begin{equation}\label{29cd}
\|u_{0}\|_{L^{\theta}(\Om)}^{\theta}>|\Om|^{\frac{1}{\theta+1}}+\left(\frac{24d_1\theta\delta|\Om|^{\frac{1}{\theta}}C}{\alpha\gamma(\theta-1)}\right)^\theta+\frac{(\theta-1)C_3\pi}{4C},
\end{equation}
and
\begin{equation}\label{29}
\|u_{0}\|_{L^{\theta}(\Om)}^{\theta}>\left(\frac{12d_1\delta C_3|\Om|^{\frac{1}{\theta}}}{\alpha\gamma}\right)^{\frac{\theta}{\theta+1}}m_1^{\theta}+\frac{(\theta-1)C_3\pi}{4C}.
\end{equation}
To see that actually
\begin{equation}\label{coms}
T^*\leq T_0:=\frac{\ln\left(\|u_{0}\|_{W^{1, \theta+1}(\Om)}^{-2(\theta+1)}+1\right)}{2C},
\end{equation}
where $C=C(r, d_1, d_2, \alpha, \beta, \gamma, \delta, \Om)>0$ is the constant arising in \eqref{2.5}, we suppose on the contrary that $T^*>T_0$, then by Theorem \ref{jbczs} we know $u\in C(\bar\Omega\times[0, T_0))\cap L^\infty([0, T_0); W^{1,q}(\Omega))$ and $v, w\in C^2(\bar\Omega\times[0, T_0))$.

Next, we aim to get a contradiction. As in \cite[Lemma 4.8]{Winkler2014How} and \cite[Lemma 4.9]{Chemotaxis}, we choose a suitable test function for the strong $W^{1, q}$-solution of problem \eqref{2}. Let $T_0\in(0, T), t_0\in(0, T_0), \delta\in(0, T_0-t_0)$ and
\begin{eqnarray*}
\chi_\delta(t)=
\left\{
\begin{array}{ll}
1,\quad &t<t_0,\\
\frac{t_0-t+\delta}{\delta},\quad & t\in[t_0, t_0+\delta],\\
0,\quad & t>t_0+\delta.
\end{array}
\right.
\end{eqnarray*}
Then for any constant $\xi>0$ and $h\in(0, 1)$ take functional
\begin{equation*}
\varphi(x,t):=\frac{\chi_\delta(t)}{h}\int_{t-h}^t(u+\xi)^{\theta-1},\quad (x,t)\in\Om\times(0,T)
\end{equation*}
in \eqref{13} we know
\begin{equation*}\begin{split}
&\frac{1}{\delta}\int_{t_0}^{t_0+\delta}\int_\Om\frac{u}{h}\int_{t-h}^t(u+\xi)^{\theta-1}-\int_\Om u_0(u_0+\xi)^{\theta-1}\\
&\quad-\int_0^T\int_\Om\chi_\delta(t)u\frac{(u+\xi)^{\theta-1}-(u(x,t-h)+\xi)^{\theta-1}}{h}\\
&=(\theta-1)\int_0^T\int_\Om\chi_\delta(t)\frac{u\nabla v}{h}\int_{t-h}^t(u+\xi)^{\theta-2}\nabla u\\
&\quad+\int_0^T\int_\Om\chi_\delta(t)\frac{ru-\mu u^\theta}{h}\int_{t-h}^t(u+\xi)^{\theta-1}.
\end{split}\end{equation*}
Use the continuity of $u$, Lebesgue's theorem and the definition of function $\chi_\delta(t)$, as $\delta\rightarrow0$ we can observe
\begin{equation}\begin{split}\label{28}
&\int_\Om\frac{u(\cdot,t_0)}{h}\int_{t_0-h}^{t_0}(u+\xi)^{\theta-1}-\int_\Omega u_0(u_0+\xi)^{\theta-1}\\
&\quad-\int_0^{t_0}\int_\Om u\frac{(u+\xi)^{\theta-1}-(u(x,t-h)+\xi)^{\theta-1}}{h}\\
&=(\theta-1)\int_0^{t_0}\int_\Om\frac{u\nabla v}{h}\int_{t-h}^t(u+\xi)^{\theta-2}\nabla u\\
&\quad+\int_0^{t_0}\int_\Om\frac{ru-\mu u^\theta}{h}\int_{t-h}^t(u+\xi)^{\theta-1}.
\end{split}\end{equation}
By \cite[Lemma 4.9]{Chemotaxis} we know
\begin{equation*}\begin{split}
&\limsup_{h\rightarrow0}\left(-\int_0^{t_0}\int_\Om u\frac{(u+\xi)^{\theta-1}-(u(x,t-h)+\xi)^{\theta-1}}{h}\right)\\
&\leq\frac{\theta-1}{\theta}\int_\Om (u_0+\xi)^{\theta}-\frac{\theta-1}{\theta}\int_\Om(u(\cdot,t_0)+\xi)^\theta-\xi\int_\Om(u_0+\xi)^{\theta-1}+\xi\int_\Om(u(\cdot,t_0)+\xi)^{\theta-1}.
\end{split}\end{equation*}
Therefore, as $h\rightarrow 0$, \eqref{28} becomes
\begin{equation*}\begin{split}
&\int_\Om u(\cdot,t_0)(u(\cdot,t_0)+\xi)^{\theta-1}+\frac{\theta-1}{\theta}\int_\Om (u_0+\xi)^{\theta}-\frac{\theta-1}{\theta}\int_\Om(u(\cdot,t_0)+\xi)^\theta\\
&\quad-\xi\int_\Om(u_0+\xi)^{\theta-1}+\xi\int_\Om(u(\cdot,t_0)+\xi)^{\theta-1}-\int_\Om u_0(u_0+\xi)^{\theta-1}\\
&\geq(\theta-1)\int_0^{t_0}\int_\Om u\nabla v(u+\xi)^{\theta-2}\nabla u-\mu\int_0^{t_0}\int_\Om u^\theta(u+\xi)^{\theta-1}.
\end{split}\end{equation*}
We therefore obtain from $\xi\rightarrow 0$ that
\begin{equation*}\begin{split}
&\int_\Om u^\theta(\cdot,t_0)+\frac{\theta-1}{\theta}\int_\Om u_0^\theta-\frac{\theta-1}{\theta}\int_\Om u^\theta(\cdot,t_0)-\int_\Om u_0^\theta\\
&\geq(\theta-1)\int_0^{t_0}\int_\Om u^{\theta-1}\nabla v\nabla u-\mu\int_0^{t_0}\int_\Om u^{2\theta-1},
\end{split}\end{equation*}
which combines integration by parts and the second and the third equations in \eqref{2} lead to
\begin{equation}\label{ni}
\begin{split}
&\frac{1}{\theta}\int_\Om u^\theta(\cdot,t_0)-\frac{1}{\theta}\int_\Om u_0^\theta\\
&\geq-\frac{(\theta-1)}{\theta}\int_0^{t_0}\int_\Om u^{\theta}\Delta v-\mu\int_0^{t_0}\int_\Om u^{2\theta-1}\\
&=\frac{(\theta-1)}{\theta}\int_0^{t_0}\int_\Om\left[\frac{d_2\alpha}{d_1\delta}u^{\theta}\Delta w+\frac{\alpha\gamma}{d_1\delta}u^{\theta+1}-\frac{\beta}{d_1}u^{\theta}v-
\frac{\theta\mu}{\theta-1}u^{2\theta-1}\right].
\end{split}
\end{equation}

Next, we aim to estimate the terms in the right hand-side of \eqref{ni}, respectively. By Lemma \ref{3.2} we know that $u$ satisfies \eqref{2.5}, so for $u^{\theta}\Delta w$, the integration by parts, the Young inequality, Lemma \ref{2ly} and \eqref{2.5} suggest that there exists arbitrary $\varepsilon_0, \varepsilon_1>0$ such that
\begin{equation}\begin{split}\label{w2}
\frac{d_2\alpha}{d_1\delta}\int_\Om u^{\theta}\Delta w&=-\frac{d_2\alpha\theta}{d_1\delta}\int_\Om u^{\theta-1}\nabla u\nabla w\\
&\geq-\varepsilon_0\int_\Om u^{\theta+1}-c(\varepsilon_0)\int_\Om |\nabla u\nabla w|^{\frac{\theta+1}{2}}\\
&\geq-\varepsilon_0\int_\Om u^{\theta+1}-\varepsilon_1c(\varepsilon_0)\int_\Om|\nabla w|^{\theta+1}-c(\varepsilon_0)c(\varepsilon_1)\int_\Om|\nabla u|^{\theta+1}\\
&\geq-\varepsilon_0\int_\Om u^{\theta+1}-\varepsilon_1c(\varepsilon_0)C_*\int_\Om u^{\theta+1}-c(\varepsilon_0)c(\varepsilon_1)m(t),
\end{split}\end{equation}
where $C_*=C(d_2, \gamma, \delta, n, \Om)$ is the constant given by Lemma \ref{2ly} and
\begin{equation*}
m(t):=\frac{1}{\sqrt{\left|\left(\|u_{0}\|_{W^{1, \theta+1}(\Om)}^{-2(\theta+1)}+1\right)e^{-2Ct}-1\right|}}.
\end{equation*}
By Lemma \ref{xs3} and \eqref{hy} we know for any $\varepsilon_2>0$ there exists a positive constant $c(\varepsilon_2)$ which depends on $\varepsilon_2, d_1, d_2, \alpha, \beta, \gamma, \delta, \theta$ such that
\begin{equation*}\begin{split}
\int_\Om v^{\theta+1}&\leq\varepsilon_2\int_\Omega u^{\theta+1}+c'(\varepsilon_2)\left(\int_\Om u\right)^{\theta+1}+c'(\varepsilon_2)\left(\int_\Om w\right)^{\theta+1}\\
&\leq\varepsilon_2\int_\Omega u^{\theta+1}+c(\varepsilon_2)m_1^{\theta+1}.
\end{split}\end{equation*}
Since $\theta\in(1, 2]$, i.e., $2\theta-1\leq\theta+1$, then for any $\varepsilon_3, \varepsilon_4>0$ the Young inequality leads to
\begin{equation}\label{v2}\begin{split}
&\int_\Om\left(-\frac{\beta}{d_1}u^{\theta}v-\frac{\theta\mu}{\theta-1}u^{2\theta-1}\right)\\
&\geq\int_\Om\left(-\varepsilon_3 u^{\theta+1}-c(\varepsilon_3)v^{\theta+1}-\varepsilon_4u^{\theta+1}-c(\varepsilon_4)\right)\\
&\geq\int_\Om\left(-\varepsilon_3u^{\theta+1}-\varepsilon_2c(\varepsilon_3)u^{\theta+1}-\varepsilon_4u^{\theta+1}\right)-\left(c(\varepsilon_4)+c(\varepsilon_2)c(\varepsilon_3)m_1^{\theta+1}\right)|\Om|.
\end{split}\end{equation}
Applying estimates \eqref{w2} and \eqref{v2} into \eqref{ni}, then
\begin{equation*}
\begin{split}
\frac{1}{\theta}\int_\Om u^\theta(\cdot,t_0)-\frac{1}{\theta}\int_\Om u_0^\theta\geq\frac{\theta-1}{\theta}\int_0^{t_0}\left[\Phi\|u\|^{\theta+1}_{L^{\theta+1}(\Om)}-\Psi\right].
\end{split}
\end{equation*}
where
\begin{equation*}\begin{split}
&\Phi=\frac{\alpha\gamma}{d_1\delta}-\varepsilon_0-\varepsilon_1c(\varepsilon_0)C_*-\varepsilon_3-\varepsilon_2c(\varepsilon_3)-\varepsilon_4,\\
&\Psi=c(\varepsilon_0)c(\varepsilon_1)m(t)+\left(c(\varepsilon_4)+c(\varepsilon_2)c(\varepsilon_3)m_1^{\theta+1}\right)|\Om|,
\end{split}\end{equation*}
Then, in turn, we take $\varepsilon_0=\varepsilon_3=\varepsilon_4=\frac{\alpha\gamma}{6d_1\delta}, \varepsilon_1=\frac{\alpha\gamma}{6d_1\delta c(\varepsilon_0)C_*}, \varepsilon_2=\frac{\alpha\gamma}{6d_1\delta c(\varepsilon_3)}$ and $\varepsilon_3=\frac{\alpha\gamma}{6d_1\delta c(\varepsilon_2)c(\varepsilon_4)}$, it is easy to see that once the constants $\varepsilon_i, i=0, \cdots, 4$, are fixed, then $c(\varepsilon_i), i=0, \cdots, 4$, are also determined by the parameters and $n, \Om$. Hence, we have $\Phi=\frac{\alpha\gamma}{6d_1\delta}$ and there exists positive constant $C_3$ which depends on $r, \mu, d_1, d_2, \alpha, \beta, \gamma, \delta, n, \Om$ such that $\Psi\leq C_3\left(m(t)+m_1^{\theta+1}\right)$, and then we obtain
\begin{equation*}
\begin{split}
\int_\Om u^\theta(\cdot,t_0)-\int_\Om u_0^\theta\geq(\theta-1)\int_0^{t_0}\left[\frac{\alpha\gamma}{6d_1\delta}\|u\|^{\theta+1}_{L^{\theta+1}(\Om)}-C_3\left(m_1^{\theta+1}+m(t)\right)\right].
\end{split}
\end{equation*}
By \eqref{29cd} we know $\|u_{0}\|_{L^{\theta}(\Om)}^{\theta}>|\Om|^{\frac{1}{\theta+1}}$, then by the H\"{o}lder inequality it is easy to obtain
\begin{equation}\label{dayu1}
\|u_{0}\|_{W^{1, \theta+1}(\Om)}^{\theta+1}>\|u_{0}\|_{L^{\theta+1}(\Om)}^{\theta+1}\geq|\Om|^{-\frac{1}{\theta}}\|u_{0}\|_{L^{\theta}(\Om)}^{\theta+1}>1,
\end{equation}
then with the help of the definitions of $m(t), T_0$ and some simple calculations we can see for $t_0\in (0, T_0)$,
\begin{equation*}\begin{split}
-\int_0^{t_0}m(t)&\geq-\int_0^{T_0}\frac{1}{\sqrt{\left|\left(\|u_{0}\|_{W^{1, \theta+1}(\Om)}^{-2(\theta+1)}+1\right)e^{-2Ct}-1\right|}}\\
&=-\frac{\arctan\left(\|u_{0}\|_{W^{1, \theta+1}(\Om)}^{-(\theta+1)}\right)}{C}\\
&\geq-\frac{\pi}{4C},
\end{split}\end{equation*}

Let $y(t)=\int_\Om u^\theta$ for $t\in(0, T_0)$, then $y(t)\in C^0([0, T_0))$ is nonnegative and it follows from the arbitrariness of $t_0$ and H\"{o}lder's inequality that
\begin{equation*}\begin{split}
y(t)&\geq y(0)-\frac{(\theta-1)C_3\pi}{4C}-(\theta-1)C_3m_1^{\theta+1}t+\frac{\alpha\gamma(\theta-1)}{6d_1\delta}|\Om|^{-\frac{1}{\theta}}\int_0^{t}y^{\frac{\theta+1}{\theta}}(\tau),
\end{split}\end{equation*}
for all $t\in(0, T_0)$, then we can infer from \eqref{29} that the condition of Lemma \ref{31} is satisfied and then
\begin{equation*}
T_0\leq\frac{12d_1\theta\delta|\Om|^{\frac{1}{\theta}}}{\alpha\gamma(\theta-1)\left[\|u_0\|^\theta_{L^\theta(\Om)}-\frac{(\theta-1)C_3\pi}{4C}\right]^{\frac{1}{\theta}}},
\end{equation*}
which combines with the definition of $T_0$ in \eqref{coms} imply that
\begin{equation*}\begin{split}
\left[\|u_0\|^\theta_{L^\theta(\Om)}-\frac{(\theta-1)C_3\pi}{4C}\right]^{\frac{1}{\theta}}&<\ln\left(\|u_{0}\|_{W^{1, \theta+1}(\Om)}^{-2(\theta+1)}+1\right)\left[\|u_0\|^\theta_{L^\theta(\Om)}-\frac{(\theta-1)C_3\pi}{4C}\right]^{\frac{1}{\theta}}\\
&\leq\frac{24d_1\theta\delta|\Om|^{\frac{1}{\theta}}C}{\alpha\gamma(\theta-1)},
\end{split}\end{equation*}
then
\begin{equation*}
\|u_0\|^\theta_{L^\theta(\Om)}<\left(\frac{24d_1\theta\delta|\Om|^{\frac{1}{\theta}}C}{\alpha\gamma(\theta-1)}\right)^\theta+\frac{(\theta-1)\pi}{4C},
\end{equation*}
this contradicts \eqref{29cd}, so we get \eqref{coms}, i.e., $T^*\leq T_0$, so the conclusion that $u$ blows up in finite time is then an immediate consequence of Theorem \ref{jbczs}.
\end{proof}

\subsection{Proof of Theorem \ref{d1q}}

Now we are at the position to show the proof of Theorem \ref{d1q}.

\begin{proof}[Proof of Theorem \ref{d1q}]
Let $(u_\epsilon, v_\epsilon, w_\epsilon)$ be the classical solution of problem \eqref{xd1} and $(u, v, w)$ be the strong $W^{1, p}$-solution of problem \eqref{2}. Assume that the conclusion of Theorem \ref{d1q} were not true, then for the solution of problem \eqref{2} there would be $M>0$ and any $\epsilon>0$ such that
\begin{equation}\label{hadcv}
u_\epsilon(x, t)\leq M,
\end{equation}
for all $(x,t)\in\Om\times(0, +\infty)$. Let $(\epsilon_j)_{j\in\mathbb{N}}\subset(0, 1)$ be a zero sequence such that $\epsilon_j\rightarrow 0$ as $j\rightarrow\infty$, so we can infer from \eqref{hadcv} and Lemma \ref{25} that
\begin{equation*}\begin{split}
&u_{\epsilon_j}\rightarrow\bar u\quad\mbox{ in}\ \ C(\bar\Om\times[0,T]),\\
&v_{\epsilon_j}\rightarrow\bar v\quad\mbox{ in}\ \ C^{2,0}(\bar\Om\times[0,T]),\\
&w_{\epsilon_j}\rightarrow\bar w\quad\mbox{ in}\ \ C^{2,0}(\bar\Om\times[0,T]),
\end{split}\end{equation*}
as $j\rightarrow\infty$, where $(\bar u, \bar v, \bar w)$ is the strong $W^{1, p}$-solution of problem \eqref{2}. Due to Lemma \ref{wyi} we know such solution is unique and then $(\bar u, \bar v, \bar w)=(u, v, w)$ and
\begin{equation*}
u(x,t)=\bar u(x,t)\leq M\quad \mbox{ in}\ \ \Omega\times(0, T).
\end{equation*}
However, since the initial data satisfies \eqref{theta}, we can see from Theorem \ref{33} that the strong $W^{1, p}$-solution $(u, v, w)$ of problem \eqref{2} blows up in finite, so there is a contradiction.
\end{proof}

\section*{Acknowledgement}

The author thanks Prof. Yasushi Taniuchi for his useful discussions on the logarithmic Sobolev inequality \cite{Kozono2000}. The author conveys thanks to Prof. Michael Winkler for his impressive work \cite{Winkler2014How} which gives many ideas for this paper. The author also thanks the reviewers for their helpful comments and suggestions which improve this paper greatly.


\end{document}